\documentclass[a4paper,twoside,DIV=12,headings=small,%
headsepline=true,toc=bib% ,captions=figureheading% ,draft
]{scrartcl}
\pdfoutput=1
\ifpdfoutput{\usepackage{cmap}}{}
\usepackage[utf8]{inputenc}
\usepackage[T1,T2A]{fontenc}
\usepackage[inner=1.9cm,outer=1.9cm,top=1.5cm,bottom=1.5cm,footskip=.5cm,headsep=.3cm]{geometry}

\def\thispapertitle {One helpful property of functions generating~P\'olya~frequency~sequences}

\usepackage{mathtools}%[mathic]
\usepackage{paralist}
\usepackage{colonequals}

\usepackage[usenames,svgnames]{xcolor}
\usepackage{float}
\usepackage%[draft]
{graphicx}

\usepackage{extarrows}
\usepackage[selected=true,shrink=10,stretch=40,verbose=true,tracking=alltext,letterspace=0
]{microtype}
\SetTracking[
    spacing = {170*,-25*,}
]{ font = {*/PTSerif-TLF/*/it/*} }{ 10 }
\usepackage{setspace}
\usepackage[permil]{overpic}

\ifpdfoutput{
\definecolor{mybrown}{HTML}{D02000}
\usepackage[colorlinks=false,linkcolor=mybrown, pdfborderstyle={/S/U/W 1},
citecolor=DarkGreen,linktocpage=true,unicode=true,pdftex,
        pdfversion=1.5,
	pdftitle={One helpful property of functions generating Polya frequency sequences},
	pdfauthor={Alexander Dyachenko},
	pdfsubject={MSC(2010) 30C15, 30D15, 30D05},
	pdfkeywords={Localization of alpha-points, Localization of zeros,
            nth root transform, R-functions, Total positivity, Polya frequency},
	pdfpagemode=UseOutlines,pdflang=en-GB,
        ]{hyperref}
        \hypersetup{
            pdftitle={\thispapertitle}
        }
}
{\usepackage[colorlinks=true,citecolor=DarkGreen,linktoc=page,unicode=true%,dvipdfmx
]{hyperref}
}
\usepackage[russian,british]{babel}
\usepackage{amsmath,amsthm}
\usepackage{amssymb,amsfonts}
\let\vvarkappa\varkappa
\usepackage[sc]{mathpazo}
\usepackage{bm}
\usepackage{caption}
\usepackage{subcaption}
\usepackage{wrapfig}
\let\varkappa\vvarkappa

\RequirePackage[scaled=0.94]{PTSansCaption}  %0.87
\RequirePackage[scaled=0.94]{PTSerifCaption} %0.87
\RequirePackage[scaled=0.94]{PTSans}
\RequirePackage[scaled=0.94]{PTSerif}

\ifpdfoutput
{\DeclareGraphicsExtensions{.pdf}}
{\DeclareGraphicsExtensions{.eps}}

\newcommand{\navy}{}%\color{Navy}}

\newcommand{\green}{\color[HTML]{509000}}
\newcommand{\gree}{}%\color[HTML]{006500}}

\newcommand{\Arg}{\operatorname{Arg}}

\newcommand{\sign}{\operatorname{sign}}

\renewcommand{\le}{\leqslant}
\renewcommand{\ge}{\geqslant}

\newcommand{\ww}{\quad\text{where}\quad}
\newcommand{\an}{\quad\text{and}\quad}

\renewcommand{\Re}{\operatorname{Re}}
\renewcommand{\Im}{\operatorname{Im}}

\newtheorem{conjecture}{\usekomafont{subparagraph}Conjecture}

\newtheorem{theorem}{\usekomafont{subparagraph}Theorem}

\newtheorem{lemma}[theorem]{\usekomafont{subparagraph}Lemma}

\newtheorem{corollary}[theorem]{\usekomafont{subparagraph}Corollary}

\theoremstyle{definition}
\newtheorem{remark}[theorem]{\usekomafont{subparagraph}Remark}

\newtheorem*{definition}{\usekomafont{subparagraph}Definition}

\numberwithin{paragraph}{section}

\newcommand{\dist}{\operatorname{dist}}

\newcommand{\myem}{\em}%{\fontfamily{PTSerifCaption-TLF}\fontshape{it}\selectfont}

\title{\thispapertitle\!%
    \thanks{This work was financially supported by the European Research Council under the European
        Union's Seventh Framework Programme (FP7/2007--2013)/ERC grant agreement no.~259173.}
}
\author{\gree\fontfamily{PTSansCaption-TLF}\selectfont\normalsize\vspace{1em}Alexander Dyachenko}
\date{\gree\fontfamily{PTSansCaption-TLF}\selectfont\vspace{-1.4em}{\small 24th June 2015}\vskip -2.5em}

\usepackage{scrpage2}
\clearscrheadfoot
\ohead{\pagemark}
\cehead{\quad A.\,Dyachenko --- \thispapertitle}
\cohead{\leftmark}
\setheadsepline{.4pt}
\deftripstyle{my_titlepagestyle}{}{}{}{}{\pagemark}{}

\addtokomafont{title}{\navy\fontfamily{PTSans-TLF}\selectfont}
\addtokomafont{section}{\navy\fontfamily{PTSansCaption-TLF}\selectfont\mdseries\large}
\addtokomafont{subsection}{\navy\fontfamily{PTSansCaption-TLF}\selectfont\mdseries}
\addtokomafont{paragraph}{\navy\fontfamily{PTSansCaption-TLF}\selectfont\mdseries}
\addtokomafont{subparagraph}{\navy\fontfamily{PTSans-TLF}\selectfont}
\addtokomafont{pageheadfoot}{\gree\fontfamily{PTSansCaption-TLF}\selectfont}
\addtokomafont{pagenumber}{\gree\fontfamily{PTSansCaption-TLF}\selectfont}
\addtokomafont{headsepline}{\gree}
\addtokomafont{footsepline}{\gree}
\setkomafont{sectionentry}{\vspace{-.7em}\fontsize{8pt}{1em}\fontfamily{PTSansCaption-TLF}\selectfont}

\makeatletter
\renewcommand{\tocbasic@@before@hook}{\vspace{-1.5em}\parskip0pt\itemsep0pt\setstretch{1}}
\makeatother

\newcommand{\Cp}{\ensuremath{\mathbb{C}_+}}
\newcommand{\Cd}{\ensuremath{\mathbb{C}_\delta}}
\newcommand{\Rc}{\ensuremath{\mathcal{R}}}
\begin{document}
\automark{section}
\maketitle
\begin{abstract}
    In this work we study the solutions of the equation~$z^pR(z^k)=\alpha$ with nonzero
    complex~$\alpha$, integer~$p,k$ and~$R(z)$ generating a (possibly doubly infinite) totally
    positive sequence. It is shown that the zeros of~$z^pR(z^k)-\alpha$ are simple (or at most
    double in the case~$\Im\alpha^k=0$) and split evenly among the sectors
    $\{\frac jk \pi\le\Arg z\le\frac {j+1}k \pi\}$, \ $j=0,\dots, 2k-1$. Our approach rests on the
    fact that $z(\ln z^{p/k}R(z) )'$ is an \Rc-function (\textit{i.e.} maps the upper half of
    the complex plane into itself).

    This result guarantees the same localization to zeros of entire
    functions~$f(z^k)+z^p g(z^k)$ and~$g(z^k)+z^{p}f(z^k)$ provided that~$f(z)$ and~$g(-z)$ have
    genus~$0$ and only negative zeros. As an application, we deduce that functions of the form
    $\sum_{n=0}^\infty (\pm i)^{n(n-1)/2}a_n z^{n}$ have simple zeros distinct in absolute value
    under a certain condition on the coefficients~$a_n\ge 0$. This includes the ``disturbed
    exponential'' function corresponding to $a_n= q^{n(n-1)/2}/n!$ when~$0<q\le 1$, as well as
    the partial theta function corresponding to $a_n= q^{n(n-1)/2}$ when
    $0<q\le q_*\approx 0.7457224107$.
\end{abstract}
\begin{center}
    {\setstretch{1}
        \begin{tabular}{@{}l|cr}
            \begin{minipage}{.305\textwidth}
                \vspace{.2em}
                %{\usekomafont{section}}
                \fontsize{8pt}{0em}\fontfamily{PTSans-TLF}\selectfont%PTSansNarrow-TLF
                Keywords:
                \begin{itemize}\itemindent0pt\itemsep1pt
                \item Localization of $\alpha$-points
                \item Localization of zeros
                \item $\mathcal R$-functions
                \item $n$th root transform
                \item Totally positive sequences
                \item P\'olya frequency sequences
                \item Partial theta function
                \end{itemize}
                % \vspace{1.85em}
                \vspace{0.5em}
                %\usekomafont{section}
                Mathematics Subject Classification (2010):
                % \vspace{.2em}\centering
                \begin{itemize}\itemindent0pt\itemsep1pt
                \item 30C15
                \item 30D15
                \item 30D05
                \end{itemize}
                \vspace{0.2em}
            \end{minipage}
          &
            \makebox[.01\textwidth]{}
          &
            \begin{minipage}{.62\textwidth}
                {\itemsep0pt\setstretch{0}
                    \tableofcontents
                    % \vspace{0.85em}
                }
            \end{minipage}
    \end{tabular}
    }
\end{center}
\section{Introduction}\label{sec:introduction}
The present paper studies quite a general equation of the form~$z^p R(z^k)$, however the simple
case~$k=2$ considered in Sections~\ref{sec:concl_for_k_2} and~\ref{sec:two-problems-sokal} has
the most interesting applications. Corollary~\ref{cr:ca2} introduces sufficient conditions on a
function of the form~$\sum_{n=0}^\infty i^{\pm \frac{n(n-1)}{2}}f_nz^n$, where~$f_0\ne 0$
and~$f_n\ge 0$ for all~$n$, which assure the simplicity of its zeros. It turns out that such
functions as
\[
\begin{aligned}
    \mathcal F(z;\pm iq) &= \sum\limits_{n=0}^\infty \frac 1{n!} (\pm iq)^{\frac{n(n-1)}2}z^{n},
    &&\text{where}\quad
    0<q\le 1,\quad\text{and}\\
    \Theta_0(z;\pm iq) &= \sum\limits_{n=0}^\infty (\pm iq)^{\frac{n(n-1)}2}z^{n},
    &&\text{where}\quad
    0<q\le q_*\approx 0.7457224107,
\end{aligned}
\]
have zeros which are simple and distinct in absolute value. The former
function~$\mathcal F(z;q)$ gives a solution to the functional-differential problem
\setstretch{1.05}
\[
\mathcal F'(z)=\mathcal F(qz),\quad \mathcal F(0)=1,
\]
while the latter is the partial theta function satisfying
\[
\Theta_0(z;q)=1+z\hspace{.7pt}\Theta_0(zq;q).
\]
The partial theta function participates in a number of beautiful Ramanujan-type relations
(\cite[Chapter~6]{AndrewsBerndt}, \cite{Warnaar}), it is related to $q$-series and some types of
modular forms. Both~$\mathcal F$ and~$\Theta_0$ appear in problems of statistics and
combinatorics (see \emph{e.g.}~\cite{Sokal,Sokal_theta}) and their zeros are the subjects of
conjectures by Alan Sokal. The details can be found in Section~\ref{sec:two-problems-sokal}.

Nevertheless, general statements offer a better insight into the problem, give an opportunity to
determine factors on which the result depends and to find possible generalizations. Their main
drawback is an excessive amount of specific cases in Sections~\ref{sec:main-theorems_PF},
\ref{sec:main-theorems} and~\ref{sec:locat-alpha-points}. To give a survey of our results, we
briefly introduce two special classes of functions and definitions of $\alpha$-sets and
$\alpha$-points.

\paragraph*{Definitions.\hspace{-.5em}}
A doubly infinite sequence $\left(\rho_n\right)_{n=-\infty}^\infty$ is called
{\myem totally positive} if all of the minors of the (four-way infinite) Toeplitz matrix
\(\left(\rho_{n-j}\right)_{n,j=-\infty}^\infty\) are nonnegative (\textit{i.e.} the matrix is
{\myem totally nonnegative}). The paper~\cite{Edrei} answers the question of convergence of the
correspondent power series~$\sum_{n=-\infty}^\infty \rho_nz^n$. Unless we
have~$\rho_n=\rho_0^{1-n}\rho_1^{n}$ for every~$n$, this series converges in some annulus to a
function of the following form
\begin{equation}\label{eq:funct_gen_dtps}
    Cz^pe^{Az+\frac{A_0}z}\cdot
    \frac{\prod_{\nu>0} \left(1+\frac{z}{a_\nu}\right)}
    {\prod_{\mu>0} \left(1-\frac{z}{b_\mu}\right)}\cdot
    \frac{\prod_{\nu>0} \left(1+\frac{z^{-1}}{c_\nu}\right)}
    {\prod_{\mu>0} \left(1-\frac{z^{-1}}{d_\mu}\right)}
\end{equation}
with absolutely convergent products, integer~\(p\) and coefficients satisfying \(A,A_0,C\ge 0\),
\(a_\nu,b_\mu,c_\nu,d_\mu>0\) for all~\(\nu,\mu\). The converse is also true: every function
with the representation~\eqref{eq:funct_gen_dtps} {\myem generates} (\textit{i.e.} its Laurent
coefficients form) a {\myem doubly infinite totally positive sequence}.

In the case of~$\cdots=\rho_{j-2}=\rho_{j-1}=0\ne\rho_j$, we assume the sequence to be
terminating on the left of~$\rho_j$ and call it {\myem totally positive}. A totally positive
sequence can be infinite when it contains no zeros to the right of~$\rho_j$ or finite otherwise.
These sequences were studied earlier than doubly infinite ones in~\cite{AESW}. They are
generated by functions of the form~\eqref{eq:funct_gen_dtps}, where the products in the last
fraction are empty and~$A_0=0$. Note that the term {\myem P\'olya frequency sequence} is often
used as a synonym for {\myem totally positive sequence}.

Herein, it is convenient to use the notion of $\alpha$-point. Given a complex number~$\alpha$,
the {\myem $\alpha$-set} of a function~$f(z)$ is the set~$\{z\in\mathbb{C}:f(z)=\alpha\}$ and
points of this set are called {\myem $\alpha$-points}. A non-constant meromorphic function can
clearly only have isolated~$\alpha$-points. We say that an $\alpha$-point~$z_*$ of a
function~$f$ {\myem has multiplicity~$n\in\mathbb{N}$} whenever
$f'(z_*)=\cdots=f^{(n-1)}(z_*)=0\ne f^{(n)}(z_*)$. The $\alpha$-point is {\myem simple} if its
multiplicity equals one.

The present work aims at describing the behaviour of $\alpha$-points of functions which can be
represented as~$z^pR(z^k)$, where~$p$ is an integer,~$k$ is a positive integer and~$R(z)$ is not
constant and generates a (possibly doubly infinite) totally positive sequence.%
\footnote{Functions of this form are the $k$th root transforms of~$z^pR^k(z)$. In the particular
    case when~$R(z)$ and~$R'(z)$ are holomorphic and nonzero at~$z=0$, the function $zR^k(z)$ is
    univalent in some disk centred at the origin. Then, $zR(z^{k})$ will be a univalent
    \emph{function with $k$-fold symmetry} in this disk in the sense that the image
    of~$zR(z^{k})$ will be $k$-fold rotationally symmetric (see e.g.~\cite[\S~2.1]{Duren} for
    the details). The term ``functions with $k$-fold symmetry'' could be good under the posed
    narrower conditions, however we study a more general case assuming no such regularity at the
    origin and allowing any integer~$p$ satisfying~$\gcd(|p|,k)=1$.} %
We confine ourselves to the case when~$\gcd(|p|,k)=1$: other cases can be treated by introducing
the variable~$\eta\colonequals z^{\gcd(|p|,k)}$. As a main tool, we use a relation of such
functions to the so-called {\myem \Rc-functions} (also known as the {\myem Pick} or
{\myem Nevanlinna} functions). By definition, \Rc-functions are the {\myem real} (\textit{i.e.}
real at every real point of continuity) functions mapping the upper half of the complex plane
into itself.

\paragraph*{Results.\hspace{-.5em}}
Our first goal is to describe the $\alpha$-set of the expression~$z^BR(z)$ in the upper half of
the complex plane~\(\Cp\), where~$R(z)$ is as above,~$B$ is real
and~$\alpha\in\mathbb{C}\setminus\{0\}$. Theorem~\ref{th:W_alpha_pts} states that if the
equation $z^BR(z)=\alpha$ has solutions in~\(\Cp\), then the \(\alpha\)-points are simple and
distinct in absolute value. The \(\alpha\)-points on the real line
% what about the origin????
\setstretch{1.2} (excepting the origin) may be either simple or double. Furthermore, for real
constants~$a$ and~$b_1\ne b_2$ the machinery developed in Lemma~\ref{lemma:simp_z} implies that
solutions to~$z^BR(z)=ae^{ib_1}$ and to~$z^BR(z)=ae^{i b_2}$ alternate when ordered in absolute
value (under the additional condition that none of them fall onto the real line --- see
Theorem~\ref{th:W_alpha_pts_order} for the details). The corresponding properties of
$\alpha$-points in the whole complex plane are described in Theorem~\ref{th:main_B} and
Remark~\ref{rem:main_B2}.

Our approach is based on Lemma~\ref{lemma:simp_z}: a function~$\psi(z)$ is univalent in the
upper half of the complex plane provided that~$z\psi'(z)$ is an \Rc-function. In fact, this
Lemma is an ``appropriate'' reformulation of classical results, however we need a construction
from its proof. Then we study the properties of~$\psi(z)$ on the real line under the additional
assumption that~$\psi(z)$ is meromorphic in~$\overline\Cp\setminus\{0\}$
(Section~\ref{sec:case-meromorphic-phi}). This assumption can be relaxed: poles can have
condensation points on the real line. However, we are interested in the narrower
case~$\psi(z)\colonequals\ln z^{B}R(z)$ with~$R(z)$ of the form~\eqref{eq:funct_gen_dtps}, which
appears in Theorem~\ref{th:W_alpha_pts} and Theorem~\ref{th:W_alpha_pts_order}. The reason is
that the function~$R(z)$ can be represented as in~\eqref{eq:funct_gen_dtps} if and only if the
sequence of its Laurent coefficients has the above property of total positivity.

The second goal of the present work is to study $\alpha$-points of~$z^pR(z^k)$, which is done by
tracking the solutions to $z^{p/k}R(z)=\alpha\cdot\exp\big(i\frac{2\pi n}{k}\big)$
and~$z^{p/k}R(z)=\overline\alpha\cdot\exp\big(i\frac{2\pi n}{k}\big)$, \ $n\in\mathbb{Z}$, under
the change of variable~$z\mapsto z^k$. This result is presented by Theorems~\ref{th:main},
\ref{th:main2},~\ref{th:meromorphic_pos_p} and~\ref{th:meromorphic_neg_p}. If we split the
complex plane into~$2k$ sectors~$Q_j= \{\frac nk \pi<\Arg z<\frac {n+1}k \pi\}$, \
$j=0,\dots, 2k-1$, then Theorem~\ref{th:main} states that for~$\Im\alpha^k\ne 0$ all
$\alpha$-points are inner points of the sectors, simple, and those in distinct sectors strictly
interlace with respect to their absolute value. Put in other words, if~$\alpha$-points
of~$z^pR(z^k)$ are denoted by~$z_i$ so that~$\cdots\le |z_{-1}|\le |z_0|\le |z_1|\le\cdots$,
then $\cdots< |z_{-1}|< |z_0|< |z_1|<\cdots$ and~$z_i\in Q_n$ implies that
$z_{i+1}, \dots, z_{i+2k-1}\notin Q_n$ and (as soon as~$R(z_{i+2k})=\alpha$) that~$z_{i+2k}\in Q_n$.
In fact, there is a formula for~$m$ such that $z_{i+1}\in Q_m$, which is
% expressed as a congruence modulo~$k$
trivial for~$p=\pm 1$ or~$k=2$. Theorem~\ref{th:main2} provides analogous properties in the
case~$\Im\alpha^k= 0$. In particular, it asserts that there are at most two $\alpha$-points
sharing the same absolute value, which are simple unless they occur at a sector boundary where
they may collapse into a double~$\alpha$-point.

In turn, Theorem~\ref{th:meromorphic_pos_p} and Theorem~\ref{th:meromorphic_neg_p} answer the
question which sector contains the $\alpha$-point that is minimal in absolute value for a
meromorphic function~$R(z)$. This automatically extends to the $\alpha$-point that is the
maximal in absolute value when~$R\left(\frac 1z\right)$ is meromorphic.

Theorems~\ref{th:main} and~\ref{th:main2}--\ref{th:meromorphic_neg_p} describe
zeros of entire functions of the form
\begin{equation}\label{eq:ent_funct_form}
    f(z^k) + z^{j} g(z^k)
    \quad\text{or}\quad
    g(z^k) + z^{j} f(z^k),
    \quad
    j,k\in\mathbb{N},
\end{equation}
where (complex) entire functions~$f(z)$ and~$g(-z)$ are of genus~$0$ and have only negative
zeros. Since $f(z^k)/f(0)$ and $g(z^k)/g(0)$ become real functions, the correspondence is
provided by
\begin{equation}\label{eq:ent_funct_corr_merom}
    f(z^k) + z^{-p} g(z^k) = 0
    \iff z^{-p}\cdot\left(g(z^k) + z^{p} f(z^k)\right) = 0
    \iff z^{p}\frac{f(z^k)/f(0)}{g(z^k)/g(0)}  = -\frac{g(0)}{f(0)}
\end{equation}
on setting~$p\colonequals\pm j$. We can allow~$f(z)$ and~$g(-z)$ to be any
\emph{functions generating totally positive sequences} up to constant complex factors. Then the
functions of the form~\eqref{eq:ent_funct_form} can be identified by the condition on their
Maclaurin or Laurent coefficients. See Section~\ref{sec:entire-functions} for further details.

Our third goal is attained in the two last sections. It consists in applying the above results
in the setting~$k=2$, which is summarized in Theorems~\ref{th:ca1}--\ref{th:ca2}. For a
(complex) entire function $H$ of the complex variable $z$ consider its decomposition into odd
and even parts such that \(H(z)=f(z^2) + z g(z^2)\). In other words, Theorem~\ref{th:ca1}
from Section~\ref{sec:concl_for_k_2} answers the following question: how are the zeros of the
function $H(z)$ distributed if the ratio~$\frac{f(z)}{g(z)}$ has only negative zeros and
positive poles? The case when the ratio~$\frac{f(z)}{g(z)}$ has only negative poles and positive
zeros is treated by Theorem~\ref{th:ca2}. The question appears to be connected to the
Hermite-Biehler theorem. This is a well-known fact asserting that if the function~$H(z)$ is a
real polynomial, then its stability%
\footnote{The polynomial is called (Hurwitz) {\myem stable} if all of its roots have negative
    real parts.} %
is equivalent to that~$f(z)$ and $g(z)$ only have simple negative interlacing%
\footnote{The zeros of two functions are called {\myem interlacing} if in between two
    consecutive zeros of the first function there is a zero of the second function and
    \emph{vice versa.}} %
zeros and~$f(0)\cdot g(0)>0$. This correspondence (expressed as conditions on the Hurwitz
matrix) is at the heart of the Routh-Hurwitz theory (see,
\emph{e.g.}~\cite[Ch.~XV]{Gantmakher},~\cite{Tyaglov,BarkovskyTyaglov,ChebMei}). With a proper
extension of the notion of stability, this criterion extends to entire (see~\cite{ChebMei}),
rational (see~\cite{BarkovskyTyaglov}) and further towards meromorphic functions. Furthermore,
if~$H(z)$ is a polynomial and we additionally allow the ratio~$\frac{f(z)}{g(z)}$ to have
positive zeros and poles, then we will obtain the ``generalized Hurwitz'' polynomials as
introduced in~\cite{Tyaglov}. In the same paper~\cite[Subsection~4.6]{Tyaglov}, its author
describes ``strange'' polynomials (related to stable polynomials) with interesting behaviour.
Item~\eqref{distr2} of our Theorem~\ref{th:ca1} explains the nature of their ``strangeness''.

\section{Connection between \texorpdfstring{$\Rc$}{R}-functions and univalent functions}
\label{sec:basic-facts-from}
Let us use the notation~``$\arg$'' for the multivalued argument function and~``$\Arg$'' for the
principal branch of argument, $-\pi<\Arg z\le\pi$ for any~$z$. We are starting from the
following short but useful observation.%
\footnote{There are many akin facts well known. For example, considering functions
    $\Phi(\zeta)\colonequals \phi(e^{-\zeta})$ gives the problem from~\cite{Wolff} but in a
    strip. %Close constructions underlay the theory of {\myem starlike functions}.
    However, we place it here since we need the relation between~$|z_1|$ and~$|z_2|$
    satisfying~\eqref{eq:lm1_cond} rather than the univalence itself.}

\begin{lemma}\label{lemma:simp_z}
    Let \(\phi\) be a function holomorphic in \(\Cp\colonequals\{z\in\mathbb{C}:\Im z>0\}\) with
    values in \(\Cp\) and let $\psi$ be a fixed holomorphic branch of
    \(\int\frac{\phi(z)}z\,dz\). Then the function~$\psi$ is univalent in~\(\Cp\). Moreover, if
    for some \(z_1,z_2\in\Cp\) we have
    \begin{equation}\label{eq:lm1_cond}
        \Re \psi(z_1) = \Re \psi(z_2)\equalscolon a
        \quad\text{and}\quad
        \Im \psi(z_1) \equalscolon b  < \widetilde b \colonequals  \Im \psi(z_2),
    \end{equation}
    then \(|z_1|<|z_2|\).
\end{lemma}

\begin{proof}
    First let us approximate the upper half-plane \Cp\ by the set
    \[
      \Cd \colonequals\big\{z\in \mathbb{C}:\ \delta<\Arg z<\pi-\delta,\ |z|>\delta\big\},
      \quad \delta>0.
    \]    
    For $z=re^{i\theta}$ we have
    \(\displaystyle
        \frac{\partial z}{\partial r} = \frac {z}{|z|}
        \)%\quad
        and
        \(\displaystyle
        \frac{\partial z}{\partial \theta} = iz
    \), so    
    \[%begin{equation*}
        r \frac{\partial}{\partial r} \Im \psi(z)
        = \Im \left(\frac{z r}{r}\psi'(z)\right) = \Im \phi(z)
        = - \frac{\partial}{\partial \theta} \Re \psi(z),
    \]%end{equation*}
    which is the matter of the Cauchy-Riemann equation. The lemma's hypothesis~$\Im \phi(z)>0$
    for $z\in\overline\Cd$ yields that
    \begin{align}\label{eq:Im_inc}
        \frac{\partial}{\partial r} \Im \psi(z) &>0
        \quad\text{and}\\ \label{eq:Re_dec}
        \frac{\partial}{\partial \theta} \Re \psi(z) &<0.
    \end{align}
    The latter inequality implies that for each $r>0$ there can be at most one value of
    $\theta\in[\delta,\pi-\delta]$ such that $\Re\psi(re^{i\theta})=a$. Moreover, the set
    $\Gamma_\delta\colonequals\big\{re^{i\theta}\in\Cd:\Re\psi(re^{i\theta})=a\big\}$ only
    consists of analytic arcs because $\Re\psi$ is a function harmonic in~$\overline{\Cd}$. In
    other words, we obtained the following.
    \begin{compactenum}[\upshape(a)]
    \item\label{item:lm1a} For every $r>0$ there is at most one point $z\in\Gamma_\delta$
        satisfying $|z|=r$. That is, the arc~$\gamma_i$ in polar coordinates~$(r,\theta)$ can be
        set by a function~$\theta(r)$.
    \end{compactenum}
    \textls[10]{Denote by $\gamma_1$, $\gamma_2$, \dots\ the connected components
        of~$\Gamma_\delta$ according to their distance to the origin, so that}
    $\dist(0,\gamma_1) \le \dist(0,\gamma_2) \le \cdots$. To count the arcs in this manner is
    possible due to the regularity of the function~$\Re\psi$ in a neighbourhood
    of~$\overline{\Cd}$ (so every bounded subdomain of~$\Cd$ contains only a finite number of
    the arcs%
    \footnote{ If it is not, then the ray $\{re^{i\theta}:r>0\}$ for appropriate
        fixed~$\theta\in\left[\delta,\pi-\delta\right]$ meets~$\overline{\Gamma_\delta}$ in an
        infinite number of points of the subdomain (it follows from~\eqref{item:lm1a}). The
        function $\Re\psi(re^{i\theta})$ is analytic in~$r>0$ (as a function of two
        variables~$\theta$ and~$r$ with $\theta$ fixed). Consequently, $\Re\psi(re^{i\theta})$
        must be constant on that ray, because it attains the same value on a point sequence
        converging to an internal point of its domain of analyticity. So, we have a
        contradiction unless~$\Gamma_\delta=\{re^{i\theta}:r>0\}$.}%
    ). It is enough to justify two additional statements, which together with~\eqref{item:lm1a}
    imply the lemma.
    \begin{compactenum}[\upshape(a)]\addtocounter{enumi}{1}
    \item\label{item:lm1b}
        On each arc $\gamma_i$, \ $i=1,2,\dots$, the value of $\Im\psi$ increases (strictly) for
        increasing~$|z|$.
    \item\label{item:lm1c}
        If we pass from $\gamma_i$ to $\gamma_{i+1}$ (due to~\eqref{item:lm1a} it corresponds to
        the grow of~$|z|$), then $\Im\psi$ cannot decrease. (In fact, we will show that these
        arcs can be connected by a line segment of $\partial\Cd$ where $\Im\psi$ increases.)
    \end{compactenum}
    To wit, the assertions~\eqref{item:lm1a}--\eqref{item:lm1c} provide that any distinct points
    of~\Cd\ giving the same $\Re\psi$ give distinct $\Im\psi$ such that the
    conditions~\eqref{eq:lm1_cond} imply~$|z_1|<|z_2|$. In particular, this yields the
    univalence of~$\psi$ in~\Cd. Furthermore, since~$\delta$ is an arbitrary positive number,
    the lemma will hold in the whole open half-plane~\Cp.
    
    For the arc~$\gamma_i$, \ $i=1,2,\dots$, consider its natural parameter~$\tau$. Orienting
    the arc according to the growth of~$r$, we obtain $\frac{\partial\tau}{\partial r}>0$. In
    addition, let us consider a coordinate $\nu$ changing in a direction orthogonal to~$\tau$,
    \textit{i.e.} such that $(\tau,\nu)$ form an orthogonal coordinate system. Then, with the help
    of the inequality~\eqref{eq:Im_inc} and one of the Cauchy-Riemann equations, we deduce that%
    \footnote{In fact we have more: $\partial\Im\psi(z) / \partial\nu=0$ implies that the
        gradient of $\Im\psi$ on~$\gamma_i$ is tangential to~$\gamma_i$.}
    \[
       0
       < \frac{\partial\Im\psi(z)}{\partial r}
       = \frac{\partial\Im\psi(z)}{\partial\tau}\frac{\partial\tau}{\partial r}
         + \frac{\partial\Im\psi(z)}{\partial\nu} \frac{\partial \nu}{\partial r}
       = \frac{\partial\Im\psi(z)}{\partial\tau}\frac{\partial\tau}{\partial r}
         \pm \frac{\partial\Re\psi(z)}{\partial\tau} \frac{\partial \nu}{\partial r}
       = \frac{\partial\Im\psi(z)}{\partial\tau}\frac{\partial\tau}{\partial r}.
    \]
    Therefore, it is true that $z_1,z_2\in\gamma_i$ and $|z_1|<|z_2|$ imply
    $\Im\psi(z_1)<\Im\psi(z_2)$, which is equivalent to~\eqref{item:lm1b}.

    Now, given two consecutive arcs~$\gamma_i$ and $\gamma_{i+1}$ consider the
    arguments~$\theta_1$ and~$\theta_2$ of their adjacent points, \textit{i.e.}
    \[
    \theta_1 \colonequals
      \lim_{|z|\to r_1:\ z\in\gamma_i}\Arg z
    \quad\text{and}\quad
    \theta_2 \colonequals
      \lim_{|z|\to r_2:\ z\in\gamma_{i+1}}\Arg z,
    \quad\text{where}\quad
    r_1=\sup_{z\in\gamma_i}|z|,\quad r_2=\inf_{z\in\gamma_{i+1}}|z|.
    \]
    The arguments can be either $\pi-\delta$ or $\delta$ since the arcs are regular, and hence
    can only end at the boundary of~\Cd. Observe that $\theta_1=\theta_2$. Indeed, let for
    example $\theta_1=\pi-\delta$. Then $\Re\psi(z)<a$ as $|z|=r_1$. However, $\theta_2=\delta$
    in its turn would imply $\Re\psi(z)>a$ when $|z|=r_2$. So, in the ``semi-annulus''
    $\{z\in\Cd:r_1<|z|<r_2\}$ there would be such $z$ that $\Re\psi(z)=a$, \textit{i.e.}
    $z\in\Gamma_\delta$ which contradicts the fact that $\gamma_i$ and $\gamma_{i+1}$ are
    consecutive arcs of~$\Gamma_\delta$.

    Since $\theta_1=\theta_2$, the
    ray~$\Theta\colonequals\left\{re^{i\theta_1},r>\delta\right\}$ meets both arcs~$\gamma_i$
    and~$\gamma_{i+1}$ in the limiting points~$r_1e^{i\theta_1}$ and~$r_2e^{i\theta_1}$,
    respectively. As a consequence, we obtain that
    $\Im\psi(r_1e^{i\theta_1})<\Im\psi(r_2e^{i\theta_1})$ since $\Im\psi$ grows everywhere
    on~$\Theta$ by the condition~\eqref{eq:Im_inc}. Then \eqref{item:lm1b} implies that
    \( \sup_{z\in\gamma_i}\Im\psi(z)\le\inf_{z\in\gamma_{i+1}}\Im\psi(z) \). Thus, the
    condition~\eqref{item:lm1c} is satisfied as well.
\end{proof}

\section{Properties of \texorpdfstring{$\alpha$}{\textalpha}-points on the real line}
\label{sec:case-meromorphic-phi}
\begin{lemma}\label{lemma:no_tri_z}
    Under the conditions of Lemma~\ref{lemma:simp_z}, let the function~\(\phi\) admit an
    analytic continuation through the interval \((x_1,x_2)\subset\mathbb{R}\setminus\{0\}\).
    Then the function~\(\psi\) defined as in Lemma~\ref{lemma:simp_z} has no triple
    \(\alpha\)-points in~\((x_1,x_2)\).
\end{lemma}
\begin{proof}
    The assertion of this lemma is exactly that $\phi(z)=z\psi'(z)$ has no double zeros on
    $(x_1,x_2)$. However, if $\phi$ could have a double zero $x_0$, then $\Im\phi(z)$ in the
    semi-disk $\{z\in\Cp:|z-x_0|<\varepsilon\ll 1\}$ must have values of both signs (since
    $\phi(z)$ is close to $(z-x_0)^2$ for such~$z$). In its turn, this
    contradicts~$\phi(\Cp)\subset\Cp$.
\end{proof}

Further in this section, we restrict the \Rc-functions~$\phi_1$,~$\phi_2$ to be meromorphic
in~$\mathbb{C}$ and real on the real line (where finite), \textit{i.e.} to have the (absolutely
convergent) Mittag-Lefler representation%
\footnote{Non-constant meromorphic functions of this form (and only of this form) map \Cp\
    into~\Cp, see~\cite[\S~V Thm.~1]{ChebMei}.}
\begin{equation}\label{eq:R_func}
    B + Az-\frac{A_0}{z}-\sum_{\nu\ne 0}\frac{z A_\nu/a_\nu}{z-a_\nu},
    \quad\text{where}\ \ 
    B,a_\nu\in\mathbb{R},\ \ 
    a_\nu\ne 0,\ \ 
    A,A_0\ge 0
    \ \ \text{and}\ \ 
    A_\nu>0
    \ \ \text{for all}\ \  \nu\ne 0,
\end{equation}
such that~$\phi_1,\phi_2\not\equiv B$. For our purposes, we need functions of more general form.
Let us take a non-constant function~$\phi(z)$ with the representation~$\phi_1(z)-\phi_2(1/z)$,
where~$\phi_1(z)$ and~$\phi_2(z)$ are as given by~\eqref{eq:R_func}. Note that both mappings
$z\mapsto\frac 1z$ and $z\mapsto -z$ map the upper half of the complex plane \Cp\ into the lower
half-plane. Therefore, $\phi$ is necessarily an \Rc-function.

\begin{remark}
    If $z\psi'(z)$ has the form~\eqref{eq:R_func}, then~$\psi(z)$ can be represented as
    \[
    \psi(z) = \int \frac{z\psi'(z)}{z}dz
    = C + B\ln z + Az + \frac {A_0}{z}
    - \sum_\nu\frac{A_\nu}{a_\nu}\ln\left(1-\frac{z}{a_\nu}\right)
    \]
    for some complex constant~$C$. This implies the equality
    \begin{equation}\label{eq:lnF}
        \Re\psi(z)
        = \Re C  + B\ln|z| + \left(A + \frac {A_0}{|z|^2}\right) \Re z
        - \sum_\nu\frac{A_\nu}{a_\nu}\ln\left|1-\frac{z}{a_\nu}\right|.
    \end{equation}
\end{remark}
\begin{remark}\label{rem:psi_1_psi_2}
    If \(z\psi'(z)=\phi(z)=\phi_1(z)-\phi_2(1/z)\), then we introduce two auxiliary
    functions~$\psi_1$ and~$\psi_2$ (single-valued in~$\overline\Cp$ where regular) so that
    $z\psi_1'(z)\equalscolon\phi_1(z)$ and $\psi(z)-\psi_1(z)\equalscolon\psi_2(z)$. These
    settings then imply $z\psi_2'(z)=-\phi_2(\frac 1z) = z^2\left(\frac 1z\right)'\!\!{}\cdot\phi_2\left(\frac 1z\right)$, that
    is~$\psi_2\left(\frac 1z\right)=\int\frac{\phi_2(z)}{z}dz$. Both $\phi_1(z)$ and $\phi_2(z)$
    satisfy~\eqref{eq:R_func}, therefore
    \[
    \Re\psi(z) = \Re\psi_1(z)+\Re\psi_2(1/z),
    \quad\text{where both~$\Re\psi_1(z)$ and~$\Re\psi_2(1/z)$ have the form~\eqref{eq:lnF}}.
    \]
    In particular, in each pole~$x_*\ne 0$ of~$\phi$ the function~$\psi$ has a logarithmic
    singularity and~$\Re\psi(z)\to +\infty\cdot x_*$ when~$z\to x_*$.
\end{remark}

\begin{lemma}\label{lemma:phi_is_real}
    If~\(x\psi'(x)=\phi(x)=\phi_1(x)-\phi_2(1/x)\), where $x\in\mathbb{R}$
    and~\(\phi_1(x)\), \(\phi_2(x)\) have the form~\eqref{eq:R_func}, then the following
    assertions are true.
    \begin{compactenum}[\upshape (a)]
    \item \label{item:a_lemma2} %
        The function~\(\Im\psi(x)\) does not change its value possibly except at the origin and
        poles of~\(\phi\).
    \item \label{item:b_lemma2} %
        Between every two consecutive negative poles~$x_2<x_1$ of~\(\phi\), there is exactly one
        local minimum of~\(\Re\psi\).
    \item \label{item:c_lemma2} %
        Between every two consecutive positive poles~$x_1<x_2$ of~\(\phi\), there is exactly one
        local maximum of~\(\Re\psi\).
    \item \label{item:d_lemma2} %
        In~\eqref{item:b_lemma2} and~\eqref{item:c_lemma2}, $x_1$ can be set to zero provided
        that $\phi$ is regular between~$0$ and~$x_2$,
        and~$\lim_{t\to+0}\big|\phi(tx_2)\big|=\infty$. In this case we
        have~$\Re\psi(tx_2)\to +\infty\cdot x_2$ as~$t\to+0$.
    \end{compactenum}
\end{lemma}
\begin{proof}
    Take a real~$x\ne 0$ such that both functions~$\phi_1(x)$ and~$-\phi_2(1/x)$ are regular.
    Since their values are real on the real line, the condition
    \[
    x\frac{\partial\Im\psi(x)}{\partial x}
    = r\frac{\partial\Im\psi(x)}{\partial r} = \Im\phi(x) = \Im\phi_1(x) - \Im\phi_2(1/x)  = 0
    \]
    is satisfied. So the assertion~\eqref{item:a_lemma2} is true.

    The function
    \( x\frac{\partial\Re\psi(x)}{\partial x} = \Re\phi(x) = \phi_1(x) - \phi_2(1/x) \) strictly
    increases form~$-\infty$ to~$+\infty$ between the points~$x_1$ and~$x_2$, and hence it
    changes its sign exactly once in the interval~$(\min(x_1,x_2),\max(x_1,x_2))$. That is,
    $\sign x\cdot\Re\psi(x)$ changes from decreasing to increasing on this interval, which is
    giving us the assertions~\eqref{item:b_lemma2} and~\eqref{item:c_lemma2} for both zero and
    nonzero~$x_1$.
    
    Suppose that the function~$\phi$ is regular between~$0$ and~$x_2$
    and~$\lim_{t\to+0}|\phi(tx_2)|$ is infinite. Then~$\phi$ increases in this interval,
    so~$\lim_{t\to+0}\phi(tx_2)=-\infty\cdot x_2$. Therefore,
    $-\psi'(tx_2)=-\frac{\phi(tx_2)}{tx_2}>\frac 1{t}$ for small enough~$t>0$ and
    \[
    \Re\psi(tx_2)
    = \Re\psi\left(\tfrac 12 x_2\right) + \int_{\frac 12 x_2}^{tx_2}\frac{\phi(x)}x dx
    = \Re\psi\left(\tfrac 12 x_2\right)
      + x_2\int_{t}^{\frac 12}\left(-\frac{\phi(sx_2)}{sx_2}\right) ds
    \to +\infty\cdot x_2 \quad\text{as}\quad t\to +0,
    \]
    which is~\eqref{item:d_lemma2}.
\end{proof}

\begin{lemma}\label{lemma:phi_is_real_XX}
    In addition to the conditions of Lemma~\ref{lemma:phi_is_real}, suppose that $\phi$ is a
    regular function in the
    interval~$\relpenalty 100 \mathfrak I=(\min\{0,x_2\},\max\{0,x_2\})\subset\mathbb{R}$, \
    $x_2$ is a pole of~$\phi$ and the limit~$\mathfrak B\colonequals\lim_{t\to+0}\phi(tx_2)$ is
    finite.\footnote{This limit exists since the function~$\phi$ increases in~$\mathfrak I$.}
    \begin{compactenum}[\upshape (a)]
    \item \label{item:a_lemma2xx}%
        \textls[30]{%
        If~$\mathfrak Bx_2>0$, then~$\Re\psi(x)$ is an increasing function
        in% the interval
~$\mathfrak I$ such that~$\Re\psi(\mathfrak I)=\mathbb{R}$, and
        furthermore,}~$\Re\psi(z)\ne\Re\psi(x)$ on condition that $|z|\le|x|$
        with~$x\in \mathfrak I$ and~$z\in\overline\Cp\setminus\{x\}$.
    \item \label{item:b_lemma2xx}%
        If~$\mathfrak Bx_2<0$, then $\Re\psi(x)$ has exactly one local extremum in~$\mathfrak I$
        and tends to~$+\infty\cdot x_2$ as~$x$ approaches~$0$ or~$x_2$.
    \item \label{item:c_lemma2xx}%
        If~$\mathfrak B=0$ then $\Re\psi(x)$ is an increasing function in~$\mathfrak I$ and the
        inequality~$\Re\psi(z)\ne\Re\psi(x)$ holds provided that~$|z|\le|x|$
        with~$z\in\overline\Cp$, \ $x\in \mathfrak I$. Moreover,
        $\lim_{t\to+0}\frac{\phi(tx_2)}{tx_2}$ is positive or~$+\infty$. If additionally $\Re\psi(tx_2)$ is
        unbounded as~$t\to+0$, then~$\Re\psi(\mathfrak I)=\mathbb{R}$.
    \end{compactenum}
\end{lemma}

\begin{proof}
    In the interval~$\mathfrak I$, the function
    \( x\frac{\partial\Re\psi(x)}{\partial x} = \phi(x)\) strictly increases, and hence
    changes its sign at most once. Therefore, $\Re\psi(x)$ has at most one local extremum:
    maximum for~$x_2<0$ and minimum for~$x_2>0$.
    Suppose that~$0<|\mathfrak B|<\infty$. Then the equality
    \( x\frac{\partial\Re\psi(x)}{\partial x} = \phi(x)\) yields the following relation
    \[
    \Re\psi(tx_2)
    = \Re\psi\left(\tfrac 12 x_2\right)
    + \int_{\frac 12 x_2}^{tx_2}\frac{\phi(x)-\mathfrak B}x dx
    + \mathfrak B\ln\frac{tx_2}{\frac 12 x_2}
    \sim \mathfrak B\ln{t} \to -\infty\cdot \mathfrak B
    \quad
    \text{when}\quad t\to+0.
    \]   
    On account of~$\Re\psi(x)\to +\infty\cdot x_2$ when~$x\to x_2$ (see
    Remark~\ref{rem:psi_1_psi_2}) this relation implies the assertion~\eqref{item:b_lemma2xx}
    and that $\Re\psi$ increases in~$\mathfrak I$ from~$-\infty$ to~$+\infty$
    if~$\mathfrak Bx_2>0$. Therefore, to obtain~\eqref{item:a_lemma2xx} it is enough to use the
    inequality
    \begin{equation} \label{eq:re_psi_ineq}
        \Re\psi(-|z|)<\Re\psi(z)<\Re\psi(|z|), \ww \Im z>0,
    \end{equation}
    which is a consequence of~\eqref{eq:Re_dec}. Indeed, if for example~$x_2<0$, then we have
    $\Re\psi(x)\le\Re\psi(-|z|)<\Re\psi(z)$ for each~$x\in\mathfrak I$ satisfying~$|x|\ge|z|$.

    The condition~$\mathfrak B=0$ implies~$\frac{\phi(x)}{x}>0$ in the interval~$\mathfrak I$,
    \textit{i.e.} that~$\Re\psi$ is growing independently of the sign of~$x_2$. The inequality
    $\lim_{t\to+0}\frac{\phi(tx_2)}{tx_2}\ne 0$ is provided by the fact that \Rc-functions
    cannot vanish faster then linearly. Furthermore, $\Re\psi$ runs through the
    whole~$\mathbb{R}$ on condition that it is unbounded near the origin, as asserted
    in~\eqref{item:c_lemma2xx}. If~$|z|\le|x|$ with~$z\in\overline\Cp$ and $x\in \mathfrak I$,
    then the inequality~\eqref{eq:re_psi_ineq} provides~$\Re\psi(z)\ne\Re\psi(x)$.
\end{proof}

\begin{remark} \label{rem:prop-a-poi-infinite} In Lemma~\ref{lemma:phi_is_real} and
    Lemma~\ref{lemma:phi_is_real_XX}, the value of~$x_2$ can be taken equal to~$+\infty$
    or~$-\infty$ at the cost of some of the conclusions. With such a choice, the
    condition~$\Re\psi(x)\to +\infty\cdot x_2$ as~$x\to x_2$ may be violated. This, in turn,
    implies that the function~$\Re\psi(x)$ in~\eqref{item:b_lemma2}, \eqref{item:c_lemma2}
    and~\eqref{item:d_lemma2} of Lemma~\ref{lemma:phi_is_real} and~\eqref{item:b_lemma2xx} of
    Lemma~\ref{lemma:phi_is_real_XX} may lose the extremum and become monotonic. In
    cases~\eqref{item:a_lemma2xx} and~\eqref{item:c_lemma2xx} of
    Lemma~\ref{lemma:phi_is_real_XX}, \ $\Re\psi(\mathfrak I)$ is only a semi-infinite interval
    of the real line, instead of the equality~$\Im\psi(\mathfrak I)=\mathbb R$.
\end{remark}

\section{Location of \texorpdfstring{$\alpha$}{\textalpha}-points
    in the closed upper half-plane}
\label{sec:main-theorems_PF}

\begin{lemma}\label{lemma:all_vals}
    % Suppose that
    Let functions~\(\phi_1(z)\), \(\phi_2(z)\)
    %have
    be of the form~\eqref{eq:R_func} and let $\psi(z)$
    % is
    be a smooth branch of~\(\int\frac{\phi_1(z)-\phi_2(1/z)}{z}dz\).
    % If the points \(z_1,z_2\in\Cp\) satisfy~\eqref{eq:lm1_cond},
    %In addition to the conditions of Lemma~\ref{lemma:simp_z}, suppose that \(\phi(z)\) has the
    % form~\eqref{eq:R_func} and
    If two points~$z_1,z_2\in\overline\Cp$ that are regular for~$\psi$ satisfy $|z_1|<|z_2|$ and
    $\Re \psi(z_1) = \Re \psi(z_2)\equalscolon a$, then
    \begin{compactenum}[\upshape(a)]
    \item \label{item:all_vals_a}%
        \( \Im\psi(z_1)\le\Im\psi(z_2)\);
    \item %
        for each~$\varrho\in(\Im\psi(z_1),\Im\psi(z_2))$ there exists $z\in\overline\Cp$ such
        that $|z_1|<|z|<|z_2|$ and~$\psi(z)=a+i\varrho$;
    \item \label{item:all_vals_c}%
        $z_1$ and~$z_2$ can be connected by a piecewise analytic curve of a finite length, on
        which~$\psi$ is smooth and~$\Im\psi(z)$ is a non-decreasing function of~$|z|$; the curve
        is a subinterval of~$\mathbb{R}$ if and only if equality holds
        in~\eqref{item:all_vals_a};
    \item \label{item:all_vals_d}%
        furthermore, equality holds in~\eqref{item:all_vals_a} if and only if
        $z_1,z_2\in\mathbb{R}$, \ $z_1\cdot z_2\ge 0$ and $\psi(z)\ne a$ for
        all~$z\in\overline\Cp$ such that~$|z_1|<|z|<|z_2|$.
    \end{compactenum}
\end{lemma}
\begin{proof}
    Recall that the function~$\phi(z)=\phi_1(z)-\phi_2(1/z)$ maps~$\Cp\to\Cp$, \textit{i.e.}
    satisfies Lemma~\ref{lemma:simp_z}. So if~$z_1,z_2\in\Cp$, then by Lemma~\ref{lemma:simp_z}
    the condition~$\Im\psi(z_1)>\Im\psi(z_2)$ induces~$|z_1|>|z_2|$,
    and~$\relpenalty 100\Im\psi(z_1)<\Im\psi(z_2)$ induces~$|z_1|<|z_2|$. As a consequence, the
    assertions~\eqref{item:all_vals_a} and~\eqref{item:all_vals_d} holds for non-real~$z_1,z_2$.

    The real part of~$\psi$ goes to~$\pm\infty$ on approaching a (nonzero) pole of~$\phi$, as
    stated in Remark~\ref{rem:psi_1_psi_2}. Consequently, it is impossible for a pole of~$\phi$
    to be a limiting point of the set
    \[
    \Gamma\colonequals\big\{re^{i\theta}\in\Cp:\Re\psi(re^{i\theta})=a,\ |z_1|<|z|<|z_2|\big\},
    \]
    so the function~$\psi$ is regular in a neighbourhood of~$\Gamma$. (Recall that~$z_1=0$ is
    allowed by the lemma's condition only if~$\psi$ is regular at the origin.) Analogously to
    the proof of Lemma~\ref{lemma:simp_z}, consider arbitrary consecutive (\textit{i.e.} there
    are no points of $\Gamma$ in between) arcs $\gamma_1$ and $\gamma_2$ such that
    $\sup_{z\in\gamma_1}|z|\equalscolon r_1\le r_2\colonequals\sup_{z\in\gamma_2}|z|$. Within
    these settings, the limits
    \[
    \zeta_1\colonequals \lim_{|z|\to r_1:\ z\in\gamma_1} z
    \quad\text{and}\quad
    \zeta_2\colonequals \lim_{|z|\to r_1:\ z\in\gamma_2} z
    \]
    exist, are real and of the same sign%
    % \footnote{\label{fn:lm1}This is induced by the condition that the arcs $\gamma_1$ and
    % $\gamma_2$ are
    % consecutive, see the proof of Lemma~\ref{lemma:simp_z}.}.
    \footnote{\label{fn:lm1}The arcs $\gamma_1$ and $\gamma_2$ are consecutive, thus one of the
        inequalities $\Re\psi(z)>a$ and~$\Re\psi(z)<a$ holds for all~$z\in\Cp$
        satisfying~$r_1<|z|<r_2$. The former inequality corresponds to the positive sign
        of~$\zeta_1,\zeta_2$, while the latter corresponds to the negative sign. See the proof
        of Lemma~\ref{lemma:simp_z} for the details.}.
    Let us show now that the function $\phi$ is regular on the line segment
    \(\mathfrak I\colonequals\big[\min\{\zeta_1,\zeta_2\},\max\{\zeta_1,\zeta_2\}\big].\)
    Indeed, in the case~$\zeta_1>0$ we have $\Re\psi(z)\to+\infty$ as~$z$ tends to a positive
    pole in contrast to $\Re\psi(z)<a$ on the whole~$\mathfrak I$.%
    \footnote{See footnote~\footref{fn:lm1}.} Similarly, $\Re\psi(z)\to-\infty$ for~$z$ tending
    to a negative pole, although $\Re\psi(z)>a$ when~$z\in \mathfrak I$ and~$\zeta_1<0$. Hence,
    the function~$\phi$ is regular on~$\mathfrak I$. By Lemma~\ref{lemma:phi_is_real},
    $\Im\psi(z)$ is constant on the segment~$\mathfrak I$ (this implies the equality
    in~\eqref{item:all_vals_d} for $z_1=\zeta_1\ne 0$ and $z_2=\zeta_2$). Summarizing, we obtain
    that the function~$\psi$ is regular on $\gamma_1\cup \mathfrak I \cup \gamma_2$, and
    $\Im\psi(z)$ is continuous and non-decreasing as~$|z|$ grows and thus attains all
    intermediate values. This reasoning is applicable for each pair of consecutive arcs
    constituting the set~$\Gamma$. That is, any two points $z_1,z_2\in\Gamma$ can be connected
    by a piecewise analytic curve of a finite length%
    \footnote{Poles of~$\phi$ can concentrate only at the origin and each interval between poles
        contains at most two ends of arcs from~$\Gamma$. That is,~$\Gamma$ contains a finite number
        of arcs. Each of the arcs has a finite length since~$\psi$ is smooth in a neighbourhood
        of~$\Gamma$. Therefore, the length of the curve is finite.},
    so that $\Im\psi$ is continuous and has a uniformly bounded growth on it. This implies the
    first two assertions of the lemma. Furthermore, we necessarily
    have~$\Im\psi(z_1)<\Im\psi(z_2)$ unless this piecewise analytic curve is a segment of the
    real line.

    It remains to check the case when~$z_1=0$. In this case,~$\psi(z)$ is regular at the origin,
    and thus (see the assertion~\ref{item:c_lemma2xx} of Lemma~\ref{lemma:phi_is_real_XX}) it is
    strictly increasing in some real interval enclosing~$z_1$. Then \eqref{eq:re_psi_ineq} shows
    that~$z_1$ is the end of some arc from~$\Gamma$. Choosing this arc as~$\gamma_1$ allows us
    to apply the previous part of the proof, and therefore yields~$\Im\psi(z_1)<\Im\psi(z_2)$.
\end{proof}

\begin{lemma} \label{lemma:equal_multiplicity}
    If~$f(z)$ is holomorphic at~$z_0$, $g(z)$ is holomorphic at~$f_0=f(z_0)$ such
    that~$g'(f_0)\ne 0$ and~$n$ is a natural number, then
    $f'(z_0)=f''(z_0)=\dots=f^{(n)}(z_0)=0$ if and only if
    \begin{equation}\label{eq:equal_multiplicity}
        \frac{dg(f(z))}{dz}\Big|_{z=z_0}
        =\frac{d^2g(f(z))}{dz^2}\Big|_{z=z_0}
        =\dots=\frac{d^ng(f(z))}{dz^n}\Big|_{z=z_0}
        =0.
    \end{equation}
    Analogously, if~$f(z)$ is holomorphic at~$z_0$ such that~$f'(z_0)\ne 0$ and~$g(z)$ is
    holomorphic at~$f_0=f(z_0)$, then the condition~\eqref{eq:equal_multiplicity} is equivalent
    to~$g'(f_0)=g''(f_0)=\dots=g^{(n)}(f_0)=0$.
\end{lemma}
\begin{proof}
    Both facts follows from solving equations provided by the chain rule sequentially:
\begin{alignat*}{2}
    &\frac{dg(f(z))}{dz}\Big|_{z=z_0}&={}& g'(f_0) f'(z_0),
    \\
    &\frac{d^2g(f(z))}{dz^2}\Big|_{z=z_0}&={}& g''(f_0) (f'(z_0))^2 + g'(f_0) f''(z_0),
    \shortintertext{\centering
        $\cdots\cdots\cdots\cdots\cdots\cdots\cdots$
    }
    &\frac{d^ng(f(z))}{dz^n}\Big|_{z=z_0}&={}& g^{(n)}(f_0) (f'(z_0))^n + \dots + g'(f_0) f^{(n)}(z_0).
\end{alignat*}
\end{proof}

\begin{remark}
    Note that \(z(\ln V(z))'\) is a meromorphic function of the form~\eqref{eq:R_func} if
    \begin{equation}\label{eq:V_is_ln_psi}
        V(z)
        = e^{Az+C+\frac{A_0}{z}} z^{B}
        \frac{\prod_{\nu>0} \left(1+\frac{z}{a_\nu}\right)^{\kappa_\nu}}
        {\prod_{\mu>0} \left(1-\frac{z}{b_\mu}\right)^{\lambda_\mu}}
        \quad
        \text{considered as a single-valued function in }\Cp,
    \end{equation}
    where~\(C\in\mathbb{C}\), \,\(B\in\mathbb{R}\), \,\(A,A_0\ge 0\), and
    \(a_\nu,\kappa_\nu,b_\mu,\lambda_\mu\) are positive reals for all \(\nu,\mu\). The
    function~$V$, as well as the more general
    \begin{equation}\label{eq:W_is_V1_V2}
        W(z)
        = e^{Az+C+\frac{A_0}{z}} z^{B}
        \frac{\prod_{\nu>0} \left(1+\frac{z}{a_\nu}\right)^{\kappa_\nu}}
        {\prod_{\mu>0} \left(1-\frac{z}{b_\mu}\right)^{\lambda_\mu}}
        \frac{\prod_{\nu>0} \left(1+\frac{1}{zc_\nu}\right)^{\widetilde\kappa_\nu}}
        {\prod_{\mu>0} \left(1-\frac{1}{zd_\mu}\right)^{\widetilde\lambda_\mu}}
        = V_1(z)\cdot V_2\left(\frac 1z\right),
    \end{equation}
    where both functions $V_1$ and $V_2$ admit the representation~\eqref{eq:V_is_ln_psi}, then
    has an analytic continuation in a neighbourhood of each its real $\alpha$-point (excluding
    the origin). This allows us to determine the multiplicity of such $\alpha$-points. Another
    straightforward fact is that both~$V$ and~$W$ never vanish outside the real line.
\end{remark}

\begin{theorem}\label{th:W_alpha_pts}
    If a function \(W\) defined in~$\overline\Cp$ has the form~\eqref{eq:W_is_V1_V2} such that
    \(W(z)\not\equiv e^Cz^B\),
    % with at least one zero or pole,
    then for any \(\alpha\in\mathbb{C}\setminus\{0\}\) the \(\alpha\)-points of~\(W(z)\)
    inside~\(\Cp\) (if they exist) are simple and distinct in absolute value from all other
    solutions to~$|W(z)|=|\alpha|$ lying in~$\overline\Cp$. The \(\alpha\)-points on the real
    line (except at~$0$) may be either simple or double.
\end{theorem}
\begin{remark}
    If~$W(z)$ is regular and nonzero for~$z=0$, then it has the form~\eqref{eq:V_is_ln_psi}
    with~$A_0=B=0$. Therefore, the equality~$W(0)=\alpha\ne 0$ necessarily gives
    $(\ln W)'(0)\ne 0$ by the case~\eqref{item:c_lemma2xx} of Lemma~\ref{lemma:phi_is_real_XX},
    that is the $\alpha$-point~$z=0$ can be only simple.
\end{remark}
\begin{proof}[Proof of Theorem~\ref{th:W_alpha_pts}]
    For~$z\psi'(z)=\phi(z)=z(\ln W (z))'$ and~$z\in\Cp$
    \[
    \phi(z)\ne 0 \implies \psi'(z)\ne 0 \implies W'(z)=\frac{\psi'(z)}{\psi(z)}\ne 0,
    \]
    which gives us that the non-real~$\alpha$-points are simple. If~$W(z)=\alpha$ then
    \begin{equation}\label{eq:abs_W_ineq}
        |W(-|z|)|<|W(z)|=|\alpha|<|W(|z|)|
        \quad
        \text{unless}
        \quad
        z=\pm|z|,
    \end{equation}
    which is in fact~\eqref{eq:re_psi_ineq}. Put in other words, solutions to~$|W(z)|=|\alpha|$
    (which include all~$\alpha$-points) in~$\overline\Cp$ have distinct absolute values. By
    Lemma~\ref{lemma:equal_multiplicity}, $\alpha$-points of~$W$ and $(\ln\alpha)$-points
    of~$\ln W$ (with the same branch of logarithm) have equal multiplicities.
    Lemma~\ref{lemma:no_tri_z} hence justify that the multiplicity of real~$\ln\alpha$-points
    of~$\psi$ is at most~$2$.
\end{proof}

\begin{theorem}\label{th:W_alpha_pts_order}
    Under the assumptions of Theorem~\ref{th:W_alpha_pts}, if~$|z_1|<|z_2|$, \ $W(z_1)=\alpha$
    and $W(z_2)=\alpha e^{i\theta}$ with a real~$\theta>0$, then for every
    \(\varrho\in(0,\theta)\) there
    exists~$z_*\in C_{12}\colonequals\{z\in\overline\Cp:|z_1|<|z|<|z_2|\}$ such
    that~\(W(z_*)=\alpha e^{i\varrho}\), unless simultaneously~\makebox{$\theta=0\pmod{2\pi k}$},
    both~$z_1$ and~$z_2$ are real of the same sign and~\(|W(z)|\ne|\alpha|\) in the
    semi-annulus~$C_{12}$.% (holding simultaneously).
\end{theorem}
\begin{proof}
    This is a straightforward corollary of Lemma~\ref{lemma:all_vals} for~$\psi(z)$ being a
    branch of~$\ln W (z)$. We only need to account that the exponential function
    maps~$\alpha+2\pi n$ for all integer~$n$ to the same point~$e^\alpha$.
\end{proof}

If the $\alpha$-set of~$W$ is not empty, then $\alpha$-points of~$W$ are assumed to be numbered
according to the growth of their absolute values,
\textit{i.e.}~\( \cdots\le|z_0|\le|z_1|\le|z_2|\le\cdots \) and
$W(z)=\alpha \iff z\in\bigcup_k z_k$. At that, each multiple $\alpha$-point we count only once.
In the sequel we consider only the case of~$C=0$: otherwise, the equality~$W(z)=\alpha$ can be
replaced with~$W(z)e^{-C}=\alpha e^{-C}$.

\begin{theorem}\label{th:main_B}
    Let~$W(z)$ be of the form~\eqref{eq:W_is_V1_V2} with natural $\kappa_\nu$,
    $\widetilde\kappa_\nu$, $\lambda_\mu$, $\widetilde\lambda_\mu$ and~$C=0$. Choose the branch of~$z^B$
    which is smooth in~$\overline\Cp\setminus\{0\}$ and positive for~$z>0$. Given a complex
    number~\(\alpha\notin\mathbb R\) such that $\alpha e^{\pm iB\pi}\notin\mathbb R$, each
    $\alpha$-point of~\(W(z)\) in~$\mathbb{C}\setminus\mathbb{R}$ is simple and distinct in
    absolute value from other~$\alpha$-points. If~$z_i,z_{i+1}$ are two consecutive points of
    the~$\alpha$-set, then~\(\Im z_i\cdot\Im z_{i+1}<0\).

    Moreover, the equations~$W(x)=\alpha$
    and~$W_\pm(-x)\colonequals \lim_{y\to \pm0}W(-x+iy)=\alpha$ have no solution for~$x>0$.
\end{theorem}
\setstretch{1.15}
Note that in the case of integer~$B$, the conditions~$\alpha e^{\pm iB\pi}\notin\mathbb R$
and~$\alpha\notin\mathbb R$ of this theorem are equivalent; furthermore, the
function~$W(x)$ is defined for~$x<0$ and equal to~$W_-(x)=W_+(x)$.
\begin{proof}
    On the one hand, for~$x>0$ the functions~$W(x)$, \ $e^{-iB\pi}W_+(-x)$
    and~$e^{iB\pi}W_-(-x)$ are real. On the other hand, both~$\alpha$ and~$\alpha e^{\pm iB\pi}$
    are non-real. Therefore, there is no solution to~$W(x)=\alpha$ and to~$W_\pm(-x)=\alpha$
    when~$x>0$. Since $W(\overline z)=\overline{W(z)}$, we can find the solutions
    to~\(W(z)=\alpha\) in the rest part of the complex plane~$\mathbb{C}\setminus\mathbb{R}$
    from the equations $W(z)=\alpha$ and $W(z)=\overline{\alpha}$ in the upper half-plane.

    Now assume that the argument of~$W$ varies in~\Cp. Theorem~\ref{th:W_alpha_pts} implies that
    all~$\alpha$-points (as well as all~$\overline\alpha$-points) of the function~$W$ are simple
    and distinct in absolute value. Furthermore, according to~\eqref{eq:abs_W_ineq} absolute
    values of $\alpha$-points and of~$\overline\alpha$-points cannot coincide (due
    to~$\alpha\ne\overline\alpha$). On account of~$\overline\alpha=\alpha e^{-2i\arg\alpha}$,
    Theorem~\ref{th:W_alpha_pts_order} (with the setting~$\theta=2\pi$) induces that if we have
    two solutions~$z_i,z_{i+k}$ to $W(z)=\alpha$, then there is a solution~$z_*$
    to~$W(z)=\overline\alpha$ such that~$|z_i|<|z_*|<|z_{i+k}|$. Conversely, between each pair
    of~$\overline\alpha$-points there is an~$\alpha$-point by the same theorem. That is, the
    absolute values of $\alpha$- and $\overline\alpha$-points interlace each other. This fact
    provides the theorem.
\end{proof}

\begin{remark}\label{rem:main_B2}
    If in Theorem~\ref{th:main_B} we take the number~\(\alpha\ne 0\) real, then the
    equations~$W(z)=\alpha$ and $W(\overline z)=\alpha$ are satisfied simultaneously. As a
    result, each $\alpha$-point of~\(W(z)\) in~$\mathbb{C}\setminus\mathbb{R}$ is simple and
    there is a unique $\alpha$-point with the matching absolute value (which is the complex
    conjugate). For an~$\alpha$-point~$z_i$ on the real line (such points are positive
    unless~$B$ is integer) there are only the following possibilities.
    \begin{compactenum}[\upshape(a)]
    \item \label{item:main_B2a} %
        The point~$z_i$ belongs to an interval between two consecutive positive poles or
        negative zeros of~$W$. If~$z_i$ is double, then the interval contains no
        other~$\alpha$-points of~$W$. If~$z_i$ is simple, then the interval contains exactly one
        another~$\alpha$-point: either~$z_{i-1}$ or~$z_{i+1}$.
    \item \label{item:main_B2b} %
        The point~$z_i$ belongs to an interval between the origin and the maximal negative zero
        or the minimal positive pole. Then exactly one another~$\alpha$-point (if~$z_i$ is
        simple) or no other $\alpha$-points (if~$z_i$ is double) lies on the same interval
        provided that~$A_0>0$ or~$Bz_i<0$ in~\eqref{eq:W_is_V1_V2}. If~$A_0=0$ and~$Bz_i\ge 0$,
        then~$z_i$ is the~$\alpha$-point minimal in absolute value and the same interval
        contains no other $\alpha$-points.
    \item \label{item:main_B2c} %
        The point~$z_i$ lies on a ray of the real line, in which~$W$ has no poles or zeros. Then
        this ray contains at most one another~$\alpha$-point of~$W$. If~$A_0=0$, \ $Bz_i\ge 0$
        and one end of this ray is the origin, then~$z_i$ is the only $\alpha$-point on the ray
        and its absolute value is minimal among all solutions to~$W(z)=\alpha$.
    \end{compactenum}
    All assertions stated here are straightforward corollaries to Lemma~\ref{lemma:phi_is_real},
    Lemma~\ref{lemma:phi_is_real_XX} and Remark~\ref{rem:prop-a-poi-infinite}. The~$\alpha$-set
    of~$W$ for $\alpha e^{\pm iB\pi}\in\mathbb R$ and~$B\notin\mathbb{Z}$ can be studied
    similarly; the main distinction is that~$W$ is not continuous on the negative semi-axis, so
    the corresponding result will be concerned with the limiting values~$W_+$ or~$W_-$.
\end{remark}

\begin{remark}
    Functions of the form~\eqref{eq:V_is_ln_psi} are meromorphic exactly when the
    exponents~$\kappa_\nu$, $\lambda_\mu$ are positive integers,~$A_0=0$ and~$B\in\mathbb{Z}$.
    These functions generate infinite totally positive sequences. Functions of the
    form~\eqref{eq:W_is_V1_V2} generate doubly infinite totally positive sequences if
    ~$\kappa_\nu, \lambda_\mu, \widetilde\kappa_\nu, \widetilde\lambda_\mu\in\mathbb{N}$ and
    $B\in\mathbb{Z}$. See the subsection ``Definitions'' of Section~\ref{sec:introduction} for
    the further details.%  Other results concerning such functions from the same point of view are
    % presented in~\cite{TyaglovDyachenko}.
\end{remark}

Hereinafter we concentrate on the case $B=\frac pk$ of~\eqref{eq:W_is_V1_V2} with positive
integers $\kappa_\nu$, $\widetilde\kappa_\nu$, $\lambda_\mu$, $\widetilde\lambda_\mu$, integer~$k\ge 2$
and~$p\ne 0$. % In other words, we assume that~$W(z^k)$ is a single-valued function
% meromorphic in~$\mathbb{C}\setminus\{0\}$.
We assume % that~\(k\) is natural and~\(p\) is integer such
that \(\gcd(|p|,k)=1\), \textit{i.e.}
the fraction~\(\frac pk\) is irreducible. The $k$th root is a $k$-valued holomorphic function on the
punctured plane~$\mathbb{C}\setminus\{0\}$. So, let $\sqrt[k]{{}\cdot{}}$ denote its branch that
is holomorphic in~$\overline\Cp\setminus\{0\}$ and maps positive semi-axis into itself. Then
\begin{equation} \label{eq:def_R_w}
    R(w)
    =(\sqrt[k]{w})^{p} e^{Aw+A_0w^{-1}}\,
    {\prod_{\nu>0} \left(1+\frac{w}{a_\nu}\right)}
    {\prod_{\nu>0} \left(1+\frac{1}{w c_\nu}\right)} \bigg/ 
    {\prod_{\mu>0} \left(1-\frac{w}{b_\mu}\right)}
    {\prod_{\mu>0} \left(1-\frac{1}{w d_\mu}\right)},
\end{equation}
in which the coefficients satisfy \(A,A_0\ge 0\) and \(a_\nu,b_\mu,c_\nu,d_\mu>0\) for
all~\(\nu,\mu\), is a single-valued meromorphic function in~$\overline\Cp\setminus\{0\}$ regular
for~$\Im w\ne 0$.

%\setstretch{1.16}
\section[Composition with \texorpdfstring{$k$th}{k-th} power function]{Composition with
    \texorpdfstring{$k$\hspace{.5pt}th}{kth} power function}
\label{sec:main-theorems}
In the current section we assume that a function \(G\not\equiv z^p\) has the representation
\begin{equation}\label{eq:G_mth}
    G(z)
    \colonequals e^{Az^k+A_0z^{-k}} z^{p}
        \frac{\prod_{\nu>0} \left(1+\frac{z^k}{a_\nu}\right)}
             {\prod_{\mu>0} \left(1-\frac{z^k}{b_\mu}\right)}
        \frac{\prod_{\nu>0} \left(1+\frac{z^{-k}}{c_\nu}\right)}
             {\prod_{\mu>0} \left(1-\frac{z^{-k}}{d_\mu}\right)}
\end{equation}
for some natural~\(k\ge 2\) and integer~\(p\), \(\gcd(|p|,k)=1\), in which the coefficients
satisfy \(A,A_0\ge 0\) and \(a_\nu,b_\mu,c_\nu,d_\mu>0\) for all~\(\nu,\mu\). As we noted above,
the case when $|p|$ and $k$ are not coprime does not need any additional study: it can be
treated by introducing the new variable~\(\eta\colonequals z^{1/\gcd(|p|,k)}\). Furthermore,
the location of zeros and poles of~$G(z)$ is clear from the expression~\eqref{eq:G_mth}, so we
concentrate on the equation~$G(z)=\alpha$ where~$\alpha\in\mathbb{C}\setminus\{0\}$.
    
For the sake of brevity denote~$e_m\colonequals\exp\left(i\frac{m}{k}\pi\right)$. The condition
\(\gcd(|p|,k)=1\) implies that
\begin{compactitem}
\item $(e_{mp})_{m=-k}^{k-1}$ is a cyclic group of order~$2k$ generated by~$e_p$ when~\(p\) is
    odd (thus $e_{mp} = e_{n}$ for $n\in\mathbb{Z}$ if and only if $mp\equiv n\pmod{2k}$);
\item $(e_{mp})_{m=0}^{k-1}$ and $(e_{mp+1})_{m=0}^{k-1}$ are two disjoint cyclic groups of
    order~$k$ generated by~$e_p$ when~\(p\) is even (the former group contains~$e_0=1$ and the
    latter one contains~$e_k=-1$).
\end{compactitem}

Denote the sectors of the complex plane with the central angle $\frac{\pi}{k}$ by
\[
    Q_s\colonequals
      \left\{z\in\mathbb{C}\setminus\{0\}:0<\Arg ze_{-s}<\frac{\pi}{k}\right\},
    \quad
    \widetilde Q_s\colonequals
      \left\{z\in\mathbb{C}\setminus\{0\}:0\le\Arg ze_{-s}<\frac{\pi}{k}\right\},
    \quad s\in\mathbb{Z},
\]
so that they are numbered in a anticlockwise direction and~$Q_s=Q_{2k+s}$, \
$\widetilde{Q}_s=\widetilde{Q}_{2k+s}$.

The equation $G(z)=\alpha$ with a fixed~$\alpha$ is equivalent to
$G(\widetilde ze_{2m})=G(\widetilde z)e_{2pm}=\alpha$ where~$m\in\mathbb{Z}$, which gives us the
following fact. (Note that we suppress the trivial case of~$G(z)$ identically equal to~$z^p$
with no special attention).
\begin{remark} \label{rem:stretch-sect-nth-power}%
    Let~$G(z)$ and~$R(w)$ be as in~\eqref{eq:G_mth} and~\eqref{eq:def_R_w}, respectively,
    $\alpha\ne 0$ and~$w\in\Cp\cup\{z>0\}$. Substituting~$z = \sqrt[k]{w}e_{2m}$
    into~\eqref{eq:G_mth} shows that if
    \begin{equation}\label{eq:for_R_w}
        R(w) = \alpha e_{-2pm},\ww m=0,\dots,k-1,
    \end{equation}
    then~$z = \sqrt[k]{w}e_{2m}\in\widetilde Q_{2m}$ solves the equation~$G(z)=\alpha$.
    Analogously, if the equality
    \begin{equation}\label{eq:for_R_w_1}
        R(w) = \overline\alpha e_{2pm},\ww m=0,\dots,k-1,
    \end{equation}
    holds for some~$m$, then~$z = \sqrt[k]{\overline w}e_{2m}\in\widetilde Q_{2m-1}$
    solves~$G(z)=\alpha$. Conversely, for each solution of~$G(z)=\alpha$ there exists an integer
    $m$ (unique under the condition~$0\le m<k$) such that~$R(z^k) = \alpha e_{-2pm}$ on
    condition that~$z^k\in\Cp\cup\{z>0\}$, or~$R(\overline z^k) = \overline\alpha e_{2pm}$ on
    condition that~$z^k\notin\Cp\cup\{z>0\}$. In this sense, the equation~$G(z)=\alpha$ can be
    replaced with the relation
    \begin{equation}\label{eq:for_R_w_2}
        R(w)\in\Omega,\ww
       \Omega\colonequals\{\alpha e_{-2pm}\}_{m=0}^{k-1}\cup\{\overline{\alpha}e_{2pm}\}_{m=0}^{k-1}
    \end{equation}
    for~$w\in\overline\Cp$, and then all~$\alpha$-points of~$G(z)$ can be determined from the
    solutions to~\eqref{eq:for_R_w_2}.
\end{remark}

\begin{remark}\label{rem:2_cases}
    \textls[10]{The relation~\eqref{eq:for_R_w_2} shows that the equation~$G(z)=\alpha$ has
        different properties depending}
    whether~$\Im\alpha^k$ is zero or not. The case
    of~$\overline{\alpha}\in\{\alpha e_{-2pm}\}_{m=0}^{k-1}$ is equivalent
    to~$\Im\alpha e_{ps}=0$ for some~$\relpenalty 100 s=0, \dots, k-1$, and hence
    to~$\Im\alpha^k = 0$. If it occurs, then the change of variable~$z\mapsto\zeta e_{s}$ in the
    equation~$G(z)=\alpha$ produces the real equation~$G(\zeta)=\alpha e_{ps}$, which has
    solutions symmetric with respect to the real line. Consequently, each solution
    to~$G(z)=\alpha$ has a pair: the reflected solution~$\overline z e_{-2s}$ with the same
    absolute value (unless they coincide). In the case
    of~$\overline{\alpha}\notin\{\alpha e_{-2pm}\}_{m=0}^{k-1}$, which is equivalent to
    $\Im\alpha^k \ne 0$, the relation~\eqref{eq:for_R_w_2} has no real solutions, and solutions
    to~\eqref{eq:for_R_w} and~\eqref{eq:for_R_w_1} have distinct absolute values, as is shown in
    Theorem~\ref{th:main_B}. Accordingly, all solutions of~$G(z)=\alpha$ are distinct in
    absolute value.

    We examine these cases in detail in Theorem~\ref{th:main2} and Theorem~\ref{th:main},
    respectively.
\end{remark}
\setstretch{1.15}

\begin{definition}
    Denote by~$\Xi$ the set of absolute values of all solutions to~$G(z)=\alpha$ with~$G$ of the
    form~\eqref{eq:G_mth}, that is
    \[\Xi\colonequals
    \{ \xi>0: \exists\theta\in(-\pi,\pi]
    \text{ such that }G(\xi e^{i\theta})=\alpha\}.
    \]
    Let~$\cdots<\xi_{i}<\xi_{i+1}<\cdots$ be the entries of~$\Xi$, such
    that~$\Xi=\{\xi_i\}_i$, and let $\dots,z_i,z_{i+1},\dots$ be the corresponding
    $\alpha$-points or, more precisely, $|z_i|=\xi_i$ and~$G(z_i)=\alpha$ for all~$i$ (that is,
    $z_i$ stands for any of the $\alpha$-points which correspond to the value of~$\xi_i$).
\end{definition}

\begin{theorem}\label{th:main}
    If~\(\Im\alpha^k\ne 0\), then for each~$i$ the $\alpha$-point~$z_i$ is simple,
    satisfies~$\Im z_i^k\ne 0$ and distinct in absolute value from other~$\alpha$-points of~$G$
    (\textit{i.e.} $G(z)=\alpha \text{ and } |z|=|z_i|\implies z=z_i$).

    If $z_i,z_{i+1}$ are two consecutive~$\alpha$-points, then the
    inclusions~$\alpha\in Q_{2q-\varkappa}$ and~$z_i\in Q_{2m-\sigma}$ with~\(q,m\in\mathbb{Z}\)
    and~\(\varkappa,\sigma\in\{0,1\}\) imply that~$z_{i+1}\in Q_{2l-1+\sigma}$, where~\(l\) is
    an integer solution of \ \(p(l+m)\equiv 2q+1-\varkappa-\sigma\pmod{k}\).
\end{theorem}

\begin{proof}
    % First note that the point~$\alpha e_{-2pm}$ run through all points of the
    % set~$\{\alpha e_{2m}\}_{m=0}^{k-1}$ when~$m$ takes the values~$0,\dots,k-1$. The elements of
    % this set, in their turn, distributed uniformly on the circle of the radius~$|\alpha|$
    % centred at the origin. That is, if we put elements of~$\{\alpha e_{-2pm}\}_{m=0}^{k-1}$ in
    % order according to their arguments, then the argument of the adjacent elements will differ
    % by~$2\frac mk \pi$. The same is true about the
    % set~$\{\overline\alpha e_{2pm}\}_{m=0}^{k-1}$.
    % Since~$\overline\alpha\notin \{\alpha e_{-2pm}\}_{m=0}^{k-1}$, for each integer~$m$
    % between~$\alpha e_{2m-2}$ and~$\alpha e_{2m}$ there exists exactly one point
    % of~$\{\overline\alpha e_{2pm}\}_{m=0}^{k-1}$.
    The expression~\eqref{eq:def_R_w} for real~$w$ yields
    that~$\Arg R(w)\in\{0,\pi,\pm\frac{p}{k}\pi\}$, and hence $R(w)\notin\Omega$.
    Consequently, all solutions to~\eqref{eq:for_R_w_2} lie in the open upper half-plane~\Cp.
    That is $G(z)\ne\alpha$ for~$\Im z^k=0$. The
    function~$R(w)$ satisfies the conditions of Theorem~\ref{th:W_alpha_pts}, thus solutions
    to~$R(w)\in\Omega$ in~\Cp\ are simple and (since~$|R(w)|=|\alpha|$ is necessary
    for~$R(w)\in\Omega$) distinct in absolute value. Therefore, all~$\alpha$-points of~$G$ are
    simple by Lemma~\ref{lemma:equal_multiplicity} and distinct in absolute value: if~$G(z)=\alpha$
    and~$|z|=|z_i|$ for some~$i$, then~$z=z_i$.

    Now let~$|z_{i}|,|z_{i+1}|\in\Xi$. Take integer numbers~$q,m,l=0,\dots, k-1$
    and~$\varkappa,\sigma,\tau=0,1$ so that~$\alpha\in Q_{2q-\varkappa}$, \
    $z_i\in Q_{2m-\sigma}$ and~$z_{i+1}\in Q_{2l-\tau}$. Then the points~$z_i\in Q_{2m}$
    and~$z_{i+1}\in Q_{2l}$ correspond to solutions of~\eqref{eq:for_R_w},
    while~$z_i\in Q_{2m-1}$ and~$z_{i+1}\in Q_{2l-1}$ correspond to solutions
    of~\eqref{eq:for_R_w_1}.

    First, assume that~$\Im\alpha^k>0$, \textit{i.e.}~$\varkappa=0$ and~$\alpha\in Q_{2q}$. Then
    the points $\alpha e_{-2pm}\in Q_{2q-2pm}$ of the set~$\Omega$ occur exactly once in each
    sector~$Q_j$ with the even indices~$j=0,2,\dots,2k-2$ when~$m$ runs over the integers
    $0,\dots,k-1$. Analogously, the points $\overline\alpha e_{2pm}\in Q_{-2q-1+2pm}$ of the
    set~$\Omega$ occur exactly once in each sector~$Q_j$ with the odd indices~$j=1,3,\dots,2k-1$ when
    $m=0,\dots,k-1$. Consequently,~$\sigma=0$ induces the equation
    $R(w_i)=\alpha e_{-2pm}\in Q_{2q-2pm}$, while~$\sigma=1$ induces
    $R(w_i)=\overline\alpha e_{-2pm}\in Q_{-2q-1+2pm}$. Combining these equalities together
    gives
    \begin{equation}\label{eq:loc_w_i}
        R(w_i)\in\Omega\cap Q_{(-1)^{\sigma}((2q+\sigma)-2pm)}
    \end{equation}
    the same reasoning for $w_{i+1}$ provides us with the condition
    \begin{equation}\label{eq:loc_w_ip1}
        R(w_{i+1})\in\Omega\cap Q_{(-1)^{\tau}((2q+\tau)-2pl)}.
    \end{equation}  
    Since $R(w_{i+1})=R(w_i) e^{i\theta}$ with an appropriate real~$\theta$, for
    all~$\rho\in(0,\theta)$ the quantity~$R(w_i) e^{i\varrho}$ cannot belong to~$\Omega$ by
    Theorem~\ref{th:W_alpha_pts_order}. However,~$\Omega$ has exactly one point in each sector
    of the complex plane, so we necessarily
    have~$R(w_{i+1})\in\Omega\cap Q_{(-1)^{\sigma}((2q+\sigma)-2pm)+1}$ from~\eqref{eq:loc_w_i}.
    Thus, on account of the relation~\eqref{eq:loc_w_ip1},
    \[
    (-1)^{\tau}((2q+\tau)-2pl)\equiv(-1)^{\sigma}((2q+\sigma)-2pm)+1 \pmod{2k}.
    \]
    Checking the parity immediately gives~$\tau=1-\sigma$. As a result,
    \[
    \begin{aligned}
        \sigma=0&\implies& (2q+1)-2pl &\equiv -(2q-2pm+1) = 2q+1-2(1+2q-pm) \pmod{2k}
        \\
        \text{and}\quad
        \sigma=1&\implies& 2q-2pl &\equiv -(2q+1-2pm)+1 = 2q-2(2q-pm) \pmod{2k}.
    \end{aligned}
    \]
    These relations imply that $2pl\equiv 2((1-\sigma)+2q-pm) \pmod {2k}$ or equivalently
    $p(l+m)\equiv 2q + 1-\sigma \pmod k$.

    \setstretch{1.25}
    \begin{figure}[H]
        \centering
        \refstepcounter{figure}\label{fig:1}\addtocounter{figure}{-1}
        \begin{overpic}[width=0.93\columnwidth]{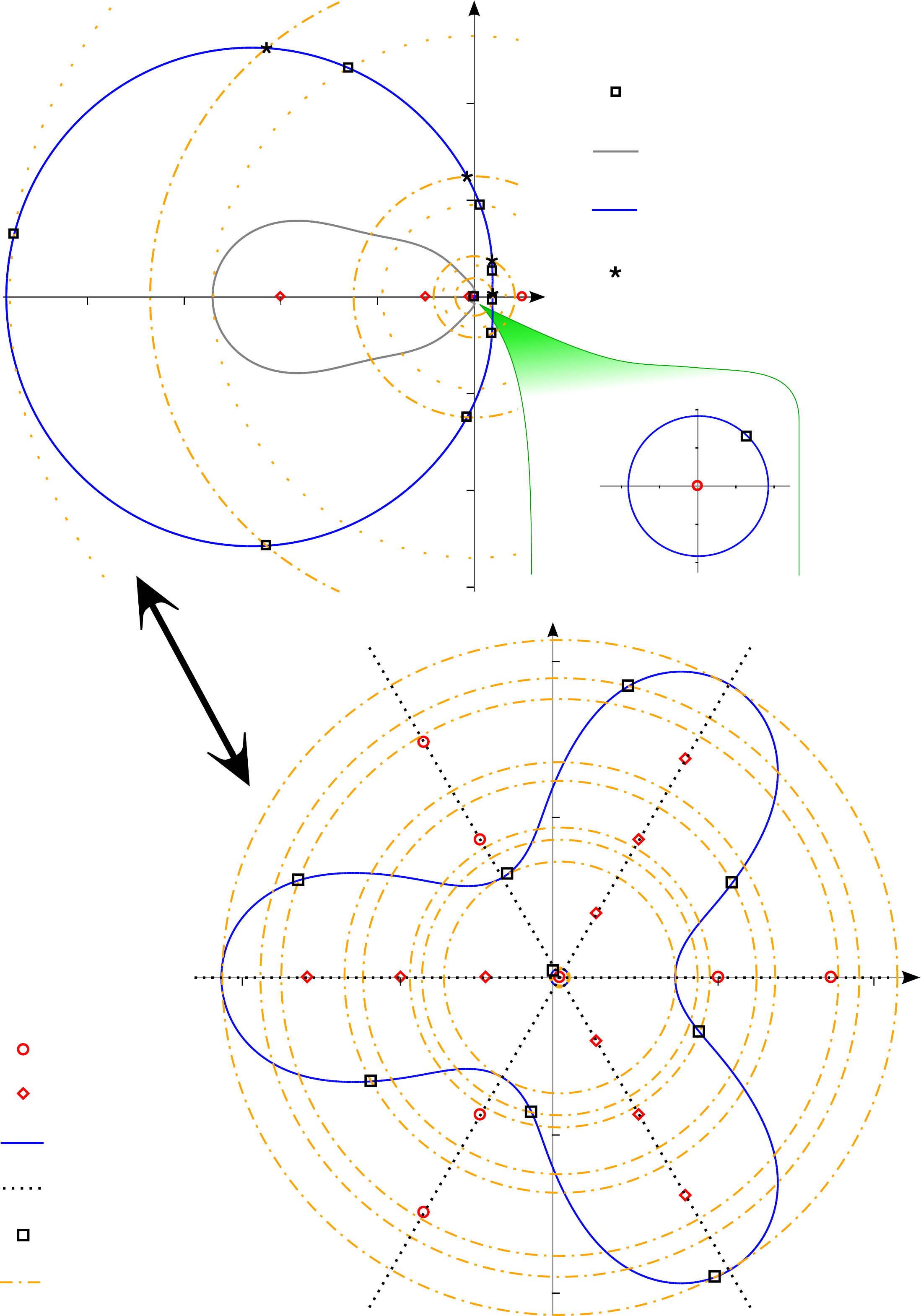}
            \put(055,755){$-8$}
            \put(128.3,755){$-6$}
            \put(201.7,755){$-4$}
            \put(275,755){$-2$}

            \put(341,916){$4$}
            \put(341,842.6){$2$}
            % \put(341,769.3){$0$}
            \put(330,696){$-2$}
            \put(330,622.7){$-4$}
            \put(330,549.3){$-6$}
            
            \put(22,785){$\Re w$}
            \put(371,987){$\Im w$}

            \put(158,493){{\LARGE$w=z^3$}}
            \put(401,695){\fontsize{9pt}{0pt}\selectfont\textls[-15]{Enlargement of the neighbourhood}}
            \put(401,679){\fontsize{9pt}{0pt}\selectfont of the origin}
            \put(492,656){\scriptsize $0.0001$}
            \put(482,598){\scriptsize $-0.0001$}
            \put(482,567){\scriptsize $-0.0002$}
            \put(545,618){\scriptsize $0.0001$}
            \put(480,618){\scriptsize $-0.0001$}
            \put(430,618){\scriptsize $-0.0002$}

            \put(450,960){Legend:}
            \put(495,925)
            {$\frac{f(w)}{\sqrt[3]{w}\cdot g(w)}=\alpha+e_m$}
            \put(495,880){$\left|\frac{f(w)}{\sqrt[3]{w}\cdot g(w)}\right|=\frac{|\alpha|}{5}$}
            \put(495,835){$\left|\frac{f(w)}{\sqrt[3]{w}\cdot g(w)}\right|=|\alpha|$}
            \put(495,789)
            {$\frac{f(w)}{\sqrt[3]{w}\cdot g(w)}=\overline\alpha +e_m$, \
                $\Im w>0$}
            \put(450,748){$e_m\colonequals\exp\frac{2\pi m i}{3}$ \ and \ $m=0,1,2$}
            
            \put(171,237){$-2$}
            \put(291,237){$-1$}
            \put(440,240){$0$}
            \put(541,237){$1$}
            \put(661,237){$2$}
            \put(430,493){$2$}
            \put(430,374){$1$}
            \put(425,128){$-1$}
            \put(425,008){$-2$}

            \put(670,270){$\Re z$}
            \put(382,518){$\Im z$}
            
            \put(0,235){Legend:}
            \put(42,197){$f(z^3)=0$}
            \put(42,164){$zg(z^3)=0$}
            \put(42,126){$\left|\frac{f(z^3)}{z g(z^3)}\right|=|\alpha|$}
            \put(42,92){$\big\{z:ze_m\in\mathbb{R}\big\}$}
            \put(42,57){$\frac{f(w)}{\sqrt[3]{w}\cdot g(w)}=\alpha$}
            \put(42,21){$\big\{z:\exists\theta\in[0,2\pi) \ \ F(z\exp{i\theta})\!{}={}\!\alpha\big\}$}
            % \put(445,750){(Here~$e_m=\exp\frac{2\pi m i}{3}$ and $m=0,1,2$.)}
            
        \end{overpic}
        \caption{Illustration to Theorems~\ref{th:main}, \ref{th:main2},~\ref{th:meromorphic_pos_p} and~\ref{th:meromorphic_neg_p}.}
        {%\vskip .3em 
            The~$\alpha$-points of the
            function~\(F(z)=\dfrac{f(z^3)}{zg(z^3)}=\dfrac{(z^3+0.1)(z^3+1)(z^3+4)}{z(z^3-1)(z^3-5)}\vphantom{\Bigg|}\),
            \ $\alpha = -1-i$, are presented in the bottom %\\[2pt]
            graph. The plot of the corresponding intermediary function~$R(w)$ is in the top
            graph. The $\alpha$-points of~$F(z)$ %\\[1pt]
            coincide with zeros of the polynomial
            \(z^9+(1+i)z^7+ 5.1 z^6 -6(1+i)z^4+ 4.5 z^3+ 5(1+i)z+0.4\).}
    \end{figure}
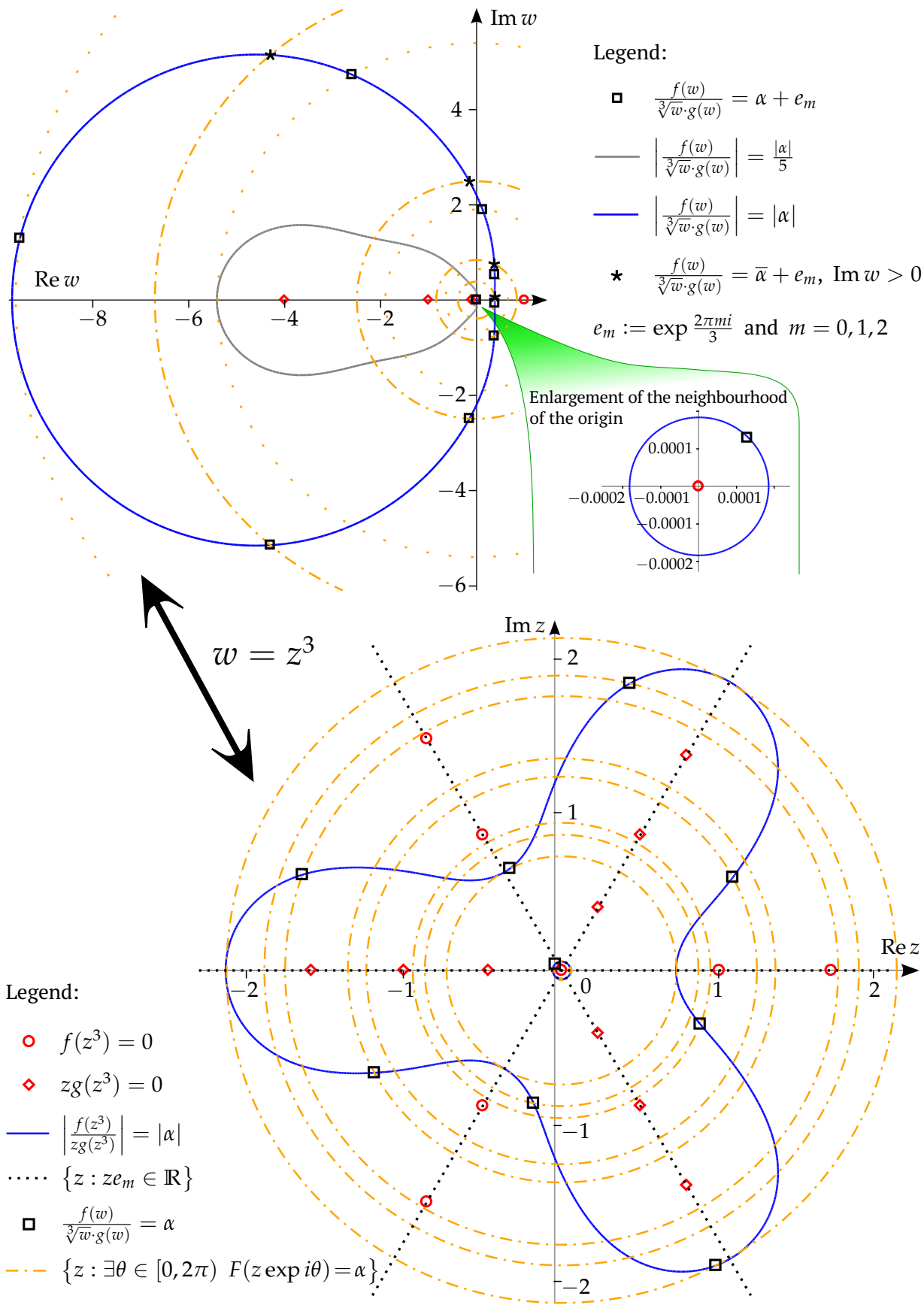\eject

    \setstretch{1.3}
    Now let $\Im\alpha^k<0$ --- that is to say, $\varkappa=1$ and~$\alpha\in Q_{2q-1}$, so
    consequently~$\alpha e_{-2pm}\in Q_{2q-1-2pm}$ and~$\overline\alpha e_{2pm}\in Q_{-2q+2pm}$.
    It implies that
    \begin{align*}
      R(w_i)&\in\Omega\cap Q_{(-1)^{\sigma}(2q-(1-\sigma)-2pm)}\an\\
      R(w_{i+1})&\in\Omega\cap Q_{(-1)^{\tau}(2q-(1-\tau)-2pl)}
                  =\Omega\cap Q_{(-1)^{\sigma}(2q-(1-\sigma)-2pm)+1}
    \end{align*}
    analogously to the case of positive~$\Im\alpha^k$. Due to the parity, we
    have~$\tau=1-\sigma$, and thus
    \[
    \begin{aligned}
        \sigma=0&\implies& -(2q-2pl) &\equiv 2q-1-2pm+1 = 2q-2pm \pmod{2k}
        \\
        \text{and}\quad
        \sigma=1&\implies& 2q-1-2pl &\equiv -(2q-2pm)+1 = -2q+2pm+1 \pmod{2k}.
    \end{aligned}
    \]
    The last two equations are equivalent to $2pl\equiv 4q-2\sigma-2pm\pmod{2k}$, which
    coincides with~$\relpenalty 1000 p(l+m)\equiv 2q-\sigma\pmod{k}$.
\end{proof}

\begin{remark}\label{sec:corr_between_rays_and_sectors}
    The rays of the line $\left\{z\in\mathbb{C}:\Im ze_s=0\right\}$, which is given
    by~$z=\overline ze_{-2s}$, can be expressed via the sectors~$Q_i$ of the complex plane by
    the formula
    \begin{equation}\label{eq:Q_es}
        \overline Q_{2m}\cap \overline Q_{-2s-2m-1}\setminus\{0\}=
        \begin{cases}
            \{z\in\mathbb{C}:ze_s>0\},&\text{if }
            m\equiv-\left\lceil \tfrac s2\right\rceil\pmod{k},\\
            \{z\in\mathbb{C}:ze_s<0\},&\text{if }
            m\equiv-\big\lceil \tfrac {s-k}2\big\rceil\pmod{k},\\
            \varnothing & \text{otherwise};
        \end{cases}
    \end{equation}
    the notation~$\lceil a\rceil$ stands for the minimal integer which is greater or equal
    to a real number~$a$.
\end{remark}

\begin{theorem}\label{th:main2}
    Let~$\Im\alpha^k=0$, \ \(\alpha\ne 0\) and the integers~$s,l$ be such
    that~$\Im\alpha e_{ps}=0$ and~$p(m-l)\equiv 1\pmod{k}$.
    \begin{enumerate}[\upshape(a)]
    \item \label{item:th_main2_a} %
        A point~$z$ satisfies the conditions $G(z)=\alpha$ and $|z|=|z_i|$ if and only
        if~$z\in\{z_i,z_i^*\}$, where~$z_i^*\colonequals\overline z_ie_{-2s}$.
        % For each~$i$, the point $z_i^*\colonequals\overline z_ie_{-2s}$ is an $\alpha$-point
        % of~$G$; $z_i$ and~$z_{i+1}$ are distinct in absolute value from other~$\alpha$-points
        % (\textit{i.e.} $G(z)=\alpha \text{ and } |z|=|z_i|\implies z\in\{z_i,z_i^*\}$).
    \item \label{item:th_main2_b} %
        The inclusion~$z_i\in Q_{2m}\cup Q_{-2s-2m-1}$ for some integer~$m$ implies that both
        $z_i^*\ne z_i$ are simple $\alpha$-points and
        $z_{i+1}\in \overline Q_{2l}\cup \overline Q_{-2s-2l-1}$ (when $|z_{i+1}|\in\Xi$).
    \item \label{item:th_main2_c} %
        The conditions~$z_i^*=z_i$ and~$\arg z_{i}=\arg z_{i+1}$ imply that both~$z_i$,
        $z_{i+1}$ are simple,~$\arg z_{i}\ne\arg z_{i-1}$ provided that~$|z_{i-1}|\in\Xi$
        and~$\arg z_{i+1}\ne \arg z_{i+2}$ provided that~$|z_{i+2}|\in\Xi$.
    \item \label{item:th_main2_d} %
        If $z_i^*=z_i$ and $\arg z_{i}\ne\arg z_{i+1}$, then $z_i$ is simple or double (which
        corresponds to, respectively, $\arg z_{i}=\arg z_{i-1}$ or $\arg z_{i}\ne\arg z_{i-1}$
        on condition that~$|z_{i-1}|\in\Xi$). Furthermore,
        $z_i\in \overline Q_{2m}\cap \overline Q_{-2s-2m-1}$ with~$m$ given by~\eqref{eq:Q_es}
        implies~$z_{i+1}\in \overline Q_{2l}\cup \overline Q_{-2s-2l-1}$.
    \item \label{item:th_main2_e} %
        If~$z_i^*=z_i$ and~$|z_{i+1}|\notin\Xi$, then the multiplicity of~$z_i$ is at most~$2$.
    \end{enumerate}        
\end{theorem}

In other words, if~\(\Im z_ie_{s}\ne 0\), then~\(z_i\) is simple, \(\Im z_i^k\ne 0\) and the
reflected point~\(z_i^*=\overline z_i e_{-2s}\) also solves~\(G(z)=\alpha\); no other
$\alpha$-points share the same absolute value. Furthermore, $z_i\in Q_{2m}$ and
$z_i^*\in Q_{-2s-2m-1}$ for some integer~$m$ (probably after exchanging
$z_i\leftrightarrow z_i^*$).

If~\(\Im z_ie_{s}=0\) \textit{i.e.} $z_i\in \overline Q_{2m}\cap \overline Q_{-2s-2m-1}$ for
some~$m$ satisfying~\eqref{eq:Q_es}, then Theorem~\ref{th:main2} asserts that~$z_i$ is simple or
double, and there are no other solutions of~\(G(z)=\alpha\) sharing the same absolute value.
If~$z_i$ is not the first or the last $\alpha$-point (with respect to the absolute value), then
either~\(z_i\) is double or exactly one another $\alpha$-point adjacent to~$z_i$ has the same
argument (in fact, it belongs to the same interval between two consequent singularities
of~\(\ln G\)).

\begin{proof}
    The equality~$G(z_i)=\alpha$ is equivalent to \(G(\overline z_ie_{-2s})=\alpha\) since
    \[
    G(\overline z_ie_{-2s})
      = \overline{G(z_ie_{2s})}=\overline{\alpha e_{2ps}}
      = \overline{\alpha e_{ps}}e_{-ps}=\alpha e_{ps}e_{-ps}=\alpha.
    \]
    Consequently,~$G(z_i)=\alpha$ if and only if~$G(z_i^*)=\alpha$,
    where~$z_i^*=\overline z_ie_{-2s}$. The points~$z_i$ and~$z_i^*$ coincide exactly
    when~$z_ie_s$ is a real number (\textit{cf.} Remark~\ref{rem:2_cases}).

    Choose the integer~$m$ satisfying~$z_i\in \widetilde Q_{2m}\cup \widetilde Q_{-2m-2s-1}$, which
    implies the same inclusion for~$z_i^*$. We constrain ourselves to the
    case~$z_i\in \overline Q_{2m}$ and thus~$z_i^*\in\overline Q_{-2m-2s-1}$: this causes no
    loss of generality since~$z_i$ and~$z_i^*$ are interchangeable with each other. The closed
    sector~$\overline Q_{2m}$ replaces~$\widetilde Q_{2m}$ since it is possible
    that~$z_i=z_i^*\in\overline Q_{2m}\cap\overline Q_{-2m-2s-1}$ (\textit{cf.}
    Remark~\ref{sec:corr_between_rays_and_sectors}). The
    point~$w_i\colonequals z_i^k\in\overline\Cp$ satisfies the
    equality~$R(w_i) = \alpha e_{-2pm}$. Conversely, if~$R(w_i) = \alpha e_{-2pm}$ then
    $z_i=\sqrt[k]{w_i}e_{2m}$ and~$z_i^*=\overline{\sqrt[k]{w_i}}e_{-2m-2s}$ are $\alpha$-points
    of~$G$.

    Now, the function~$R(w)$ satisfies the conditions of Theorem~\ref{th:W_alpha_pts}.
    Therefore, solutions of~$R(w)\in\Omega$ %=\{\alpha_{-2pm}\}_{m=0}^{k-1}
    in the closed upper half-plane~$\overline\Cp$ are distinct in absolute value; those in~\Cp\
    are additionally simple, and those on the real line are simple or double. In particular,
    if~$R(w)\in\Omega$ and~$|w|=|w_i|$ then~$w=w_i$, which implies the
    assertion~\eqref{item:th_main2_a}. Moreover, by Lemma~\ref{lemma:equal_multiplicity} the
    multiplicity of~$z_i$ equals one in the assertion~\eqref{item:th_main2_b} and is at most two
    in the assertions~\eqref{item:th_main2_c}--\eqref{item:th_main2_e} (and
    therefore~\eqref{item:th_main2_e} is proved).

    Now let~$|z_{i+1}|\in\Xi$. Then, analogously to~$z_i$, the points~$z_{i+1}$
    and~$\overline z_{i+1}e_{-2s}$ are the only two solutions of the equation~$G(z)=\alpha$
    which satisfy~$|z|=|z_{i+1}|$. Furthermore, we can assume that~$z_{i+1}\in\overline Q_{2l}$
    for some integer~$l$ without loss of generality.
    Then~$w_{i+1}\colonequals z_{i+1}^k\in\overline\Cp$
    implies~$z_{i+1}=\sqrt[k]{w_{i+1}}e_{2l}$, and
    consequently~$R(w_{i+1})=\alpha e_{-2pl}$.

    \textls[20]{Observe that the points~$w_i,w_{i+1}\in\overline\Cp$ satisfy the
    conditions~$|w_i|<|w_{i+1}|$, \ $R(w_i)=\alpha e_{-2pm}$}
    and~$R(w_{i+1})=\alpha e_{-2pl}=\alpha e_{-2pm+2\delta}$ for appropriate integers~$m,\delta$
    and the quantity~$\alpha e_{2pm} e^{i\varrho}$ cannot belong to~$\Omega$ for
    all~$\rho\in(0,\frac{2\delta\pi}k)$. By Theorem~\ref{th:W_alpha_pts_order}, this is possible
    only in two cases: if~$\delta=1$ or if simultaneously:~$\delta=0$, \
    $\Arg w_i=\Arg w_{i+1}\in\{0,\pi\}$ and~$|R(w)|\ne|\alpha|$ provided
    that~$|w_i|<|w|<|w_{i+1}|$. In the former case, we necessarily obtain the
    equation~$-2pl\equiv -2pm+2\delta\pmod{2k}$ with respect to the unknown~$l$, that
    is~$p(m-l)\equiv \delta=1\pmod{k}$. This proves the assertion~\eqref{item:th_main2_b}.

    Lemma~\ref{lemma:phi_is_real}, Lemma~\ref{lemma:phi_is_real_XX} and
    Remark~\ref{rem:prop-a-poi-infinite} together yield that, on each subinterval of the real
    line containing no singularities of~$\ln R$, \ $\arg R(w)$ is constant and there are at most
    two solutions (counting with multiplicities) to~$|R(w)|=|\alpha|$. Moreover, there are
    exactly one double or two simple solutions provided that both ends of the subinterval are
    positive poles or negative zeros of~$R$. Therefore,~$\arg z_i=\arg z_{i+1}$ (this condition
    corresponds to the case~$\delta=0$ above) implies that both~$z_i$ and~$z_{i+1}$ are simple.

    Let~$\arg z_i=\arg z_{i+1}$. The assertion~\eqref{item:th_main2_c} will be proved if we show
    that~$\arg z_{i}\ne\arg z_{i-1}$ when~$|z_{i-1}|\in\Xi$ and
    that~$\relpenalty 100 \arg z_{i+1}\ne\arg z_{i+2}$ when~$|z_{i+2}|\in\Xi$. We do it by
    contradiction: suppose that, for example,~$\relpenalty 100 \arg z_{i}=\arg z_{i-1}$. Then
    $R(w_{i-1})=R(w_{i})=\alpha e_{-2pm}$ with~$w_{i-1}\colonequals z_{i-1}^k$. On the one hand,
    the interval between~$w_{i-1}$ and~$w_{i}$ contains at least one singularity of~$\ln R$, and
    thus it cannot be a curve provided by~\eqref{item:all_vals_c} of Lemma~\ref{lemma:all_vals}.
    Along that curve,~$\arg R(w)$ is continuous and must have a nonzero increment%
    \footnote{More specifically, the difference $\arg R(w_{i}) - \arg R(w_{i-1})$ is positive
        for any fixed branch of~$\arg R$, which is continuous on the curve.}.
    Then the increment is at least~$2\pi$ due to the condition~$R(w_{i-1})=R(w_{i})$. On the
    other hand, Theorem~\ref{th:W_alpha_pts_order} then implies that there exists a
    point~$w_*\in\overline\Cp$ such that~$R(w_*)\in\Omega$ and~$|w_{i-1}|<|w_*|<|w_i|$
    since~$\Omega$ contains at least two points. Therefore,~$z_*=\sqrt[k]{w_*}e_{2p\widetilde m}$
    with a proper choice of~$\widetilde m$ satisfies the relations~$G(z_*)=\alpha$
    and~$|z_{i-1}|<|z_*|<|z_{i}|$, which contradicts to the definitions of the points~$z_{i-1}$
    and~$z_i$. As a consequence,~$\arg z_{i}\ne\arg z_{i-1}$. On condition
    that~$|z_{i+2}|\in\Xi$, the inequality~$\arg z_{i+1}\ne\arg z_{i+2}$ can be obtained
    analogously. Therefore,~\eqref{item:th_main2_c} is true. Furthermore, the similar proof
    gives us that if~$\arg z_i\ne\arg z_{i+1}$ (which corresponds to the case~$\delta=1$ above),
    $|z_{i-1}|\in\Xi$ and~$z_i$ is a double $\alpha$-point, then~$\arg z_i\ne\arg z_{i-1}$.

    The case when the $\alpha$-point~$z_i$ is real and simple, $\arg z_{i-1}\ne\arg z_i$
    and~$\arg z_i\ne\arg z_{i+1}$ is impossible. Indeed, let~$w_{i-1}$ be introduced in a way
    similar to~$w_i$ and~$w_{i+1}$. Suppose firstly that~$w_i>0$. Then we
    have~$\big|R(|w_{i\pm1}|)\big|<\big|R(w_{i\pm1})\big|=\big|R(w_i)\big|$ according to the
    inequality~\eqref{eq:abs_W_ineq}. If~$\ln R$ is regular in the interval
    $\mathfrak{I}\colonequals(|w_{i-1}|,|w_{i+1}|)$, then there exists a point~$w_*$ of this
    interval such that~$\big|R(w_*)\big|=|\alpha|$ and, since~$w_i$ is simple,~$w_*\ne w_i$. At
    the same time,~$\arg R$ does not change in~$\mathfrak{I}$, which gives us the
    contradiction~$R(w_*)=R(w_i)=\alpha$. If~$\ln R$ has singularities in the
    interval~$\mathfrak{I}$, then, instead of~$\mathfrak{I}$, it is enough to consider the
    maximal subinterval of~$\mathfrak{I}$ containing~$w_i$, in which~$\ln R$ is regular. The
    case~$w_i<0$ is proved analogously with the help of the
    inequality~$\big|R(-|w_{i\pm1}|)\big|>\big|R(w_{i\pm1})\big|=\big|R(w_i)\big|$, which is
    provided by~\eqref{eq:abs_W_ineq}. The assertion~\eqref{item:th_main2_d} holds, so the
    theorem is completely proved.
\end{proof}

\section{Location of the \texorpdfstring{$\alpha$}{\textalpha}-point that
is minimal or maximal in absolute value}
\label{sec:locat-alpha-points}
Let a function \(F\) have the form
\begin{equation}\label{eq:F_mth}
    F(z)
    \colonequals z^{p} e^{Az^k}\,
    % \frac
    \frac{\displaystyle\prod_{\nu=1}^{\omega_1} \left(1+\frac{z^k}{a_\nu}\right)}%\ \Bigg/\ 
    {\displaystyle\prod_{\mu=1}^{\omega_2} \left(1-\frac{z^k}{b_\mu}\right)}, \qquad F(z)\not\equiv z^p,
\end{equation}
where $k$ and~$p$ are integer such that~$k\ge 2$ and~\(\gcd(|p|,k)=1\),
\(0<\omega_1,\omega_2\le +\infty\), \ \(A\ge 0\) and \(a_\nu,b_\mu>0\) for all~\(\nu,\mu\). Such
functions are the particular case of~\eqref{eq:G_mth} and, therefore, satisfy conditions of
Theorem~\ref{th:main} and Theorem~\ref{th:main2}. The next two theorems reveal another property
of the~$\alpha$-set of~$F$. Assuming that the $\alpha$-set is nonempty, they answer which of the
sectors contains the $\alpha$-point (or $\alpha$-points) of the function~$F$ that is minimal in
absolute value.

\begin{theorem}\label{th:meromorphic_pos_p}
    Consider a complex number \(\alpha\ne 0\) and a function \(F\) of the form~\eqref{eq:F_mth}
    with~$p>0$. Let $q=0,\dots,k-1$ and $\varkappa=0,1$ be chosen so that
    $\alpha\in\widetilde Q_{2q-\varkappa}$, and the integer~$m$ be such that \(pm\equiv q\pmod{k}\).
    
    If $\alpha^k\not<0$, then the $\alpha$-point~$z_*$ of~\(F(z)\) closest to the origin is
    simple and distinct in argument and absolute value from the succeeding $\alpha$-points.
    Moreover, $\alpha\in Q_{2q-\varkappa}$ implies $z_*\in Q_{2m-\varkappa}$. If
    $\alpha e_{-2q}>0$, then~$z_*e_{-2m}>0$.

    If~$\alpha^k<0$, that is $\alpha e_{-2q+1}>0$, then the two zeros of~$F(z)-\alpha$ closest
    to the origin (counting double zeros as two) are equal in absolute value or in argument. The
    latter case is possible only when~$p=1$; if it occurs, both zeros belong to the ray
    $\{ze_{-2m+1}>0\}$. In the former case, one of them belongs to~$Q_{2m-1}$ and another
    belongs to~$Q_{2\widetilde m} = Q_{-2m-2s}$ where~$\widetilde m$ satisfies
    \(p\widetilde m \equiv q - 1\pmod{k}\) and~$s$ is introduced in Remark~\ref{rem:2_cases}.
\end{theorem}
\begin{proof}
    Let~$z_0$ denote the solution of the equation~$F(z)=\alpha$ that is minimal in absolute
    value. Consider the corresponding point~$w_0\in\overline\Cp$ determined by~$w_0=z_0^k$
    if~$\Im z_0^k\ge 0$ and by~$w_0=\overline z_0^k$ if~$\Im z_0^k\le 0$. Recall that (see
    Remark~\ref{rem:stretch-sect-nth-power}) the equality~$F(z_0)=\alpha$ is equivalent
    to~$R(w_0)\in\Omega$, where the function $R(w)=F(\sqrt[k]{w})$ is defined
    in~$\overline\Cp\setminus\{0\}$ by the equality~\eqref{eq:def_R_w}
    and~$\Omega=\{\alpha{}e_{2m}\}_{m=0}^{k-1}\cup\{\overline\alpha{}e_{2m}\}_{m=0}^{k-1}$.

    The key moment here is to implement the assertion~\eqref{item:a_lemma2xx} of
    Lemma~\ref{lemma:phi_is_real_XX} with the setting~$\psi=\ln R$. It implies that the
    point~$w_*$ of set
    $\Gamma\colonequals\left\{w\in\overline\Cp:\left|R(w)\right|=|\alpha|\right\}$ which is the
    closest to the origin satisfies~$0<w_*<b_1$, because
    \[
    \mathfrak B=\lim_{t\to+0}\phi(tb_1)=\lim_{t\to+0}t\psi'(t)=\frac pk>0.
    \]
    (Moreover, $R(w)$ cannot be uniformly bounded in~$\{w:w>0\}$, so~$w_*$ necessarily exists.)
    Note that~$R(w)$ has the form~\eqref{eq:def_R_w}, so $R(w_*)>0$, that is $R(w_*)=|\alpha|$.
    Putting $z_*\colonequals\sqrt[k]{w_*}e_{-2m}$ we obtain $F(z_*)=|\alpha| e_{2pm}$. As
    suggested by the theorem's statement, the integer~$m$ satisfies $pm\equiv q\pmod{k}$.
    Consequently, if $\alpha e_{-2q} = |\alpha|>0$ then the point~$z_0\colonequals z_*$
    satisfying the inequality~$z_0e_{-2m}>0$ is the zero of~$F(z)-\alpha$ we are looking for; it
    is simple by Lemma~\ref{lemma:equal_multiplicity}. (The example is given in
    Figure~\ref{fig:0a}, $\alpha=e^{i2\pi/3}$.)

    Suppose now that~$\alpha\in Q_{2q}$. Then the
    increment~$\arg R(w_0) -\arg R(w_*) = \Arg(\alpha e_{-2q})$ is positive and less
    than~$\frac{\pi}k$. Theorem~\ref{th:W_alpha_pts_order} implies that~$R(w)\notin\Omega$ on
    condition that $|w_*|<|w|<|w_0|$, and moreover, $R(w_0)=\alpha{}e_{-2q}=\alpha{}e_{-2pm}$.
    Therefore,~$z_0=\sqrt[k]{w_0}e_{2m}\in Q_{2m}$ is the required~$\alpha$-point.
    \setstretch{1.25}
    
    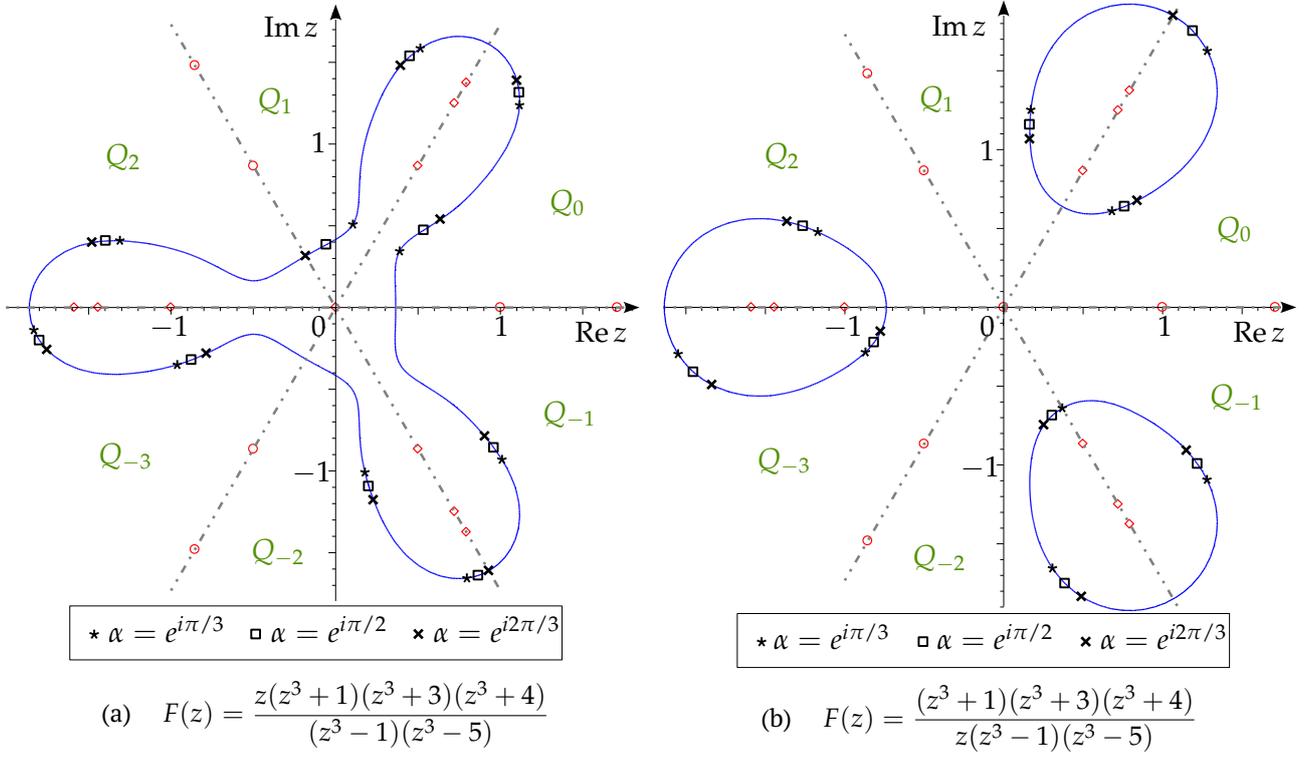
\begin{figure}[H]
        \begin{subfigure}{.495\linewidth}
            \centering
            \begin{overpic}[width=0.98\columnwidth%,grid,tics=25
                ]{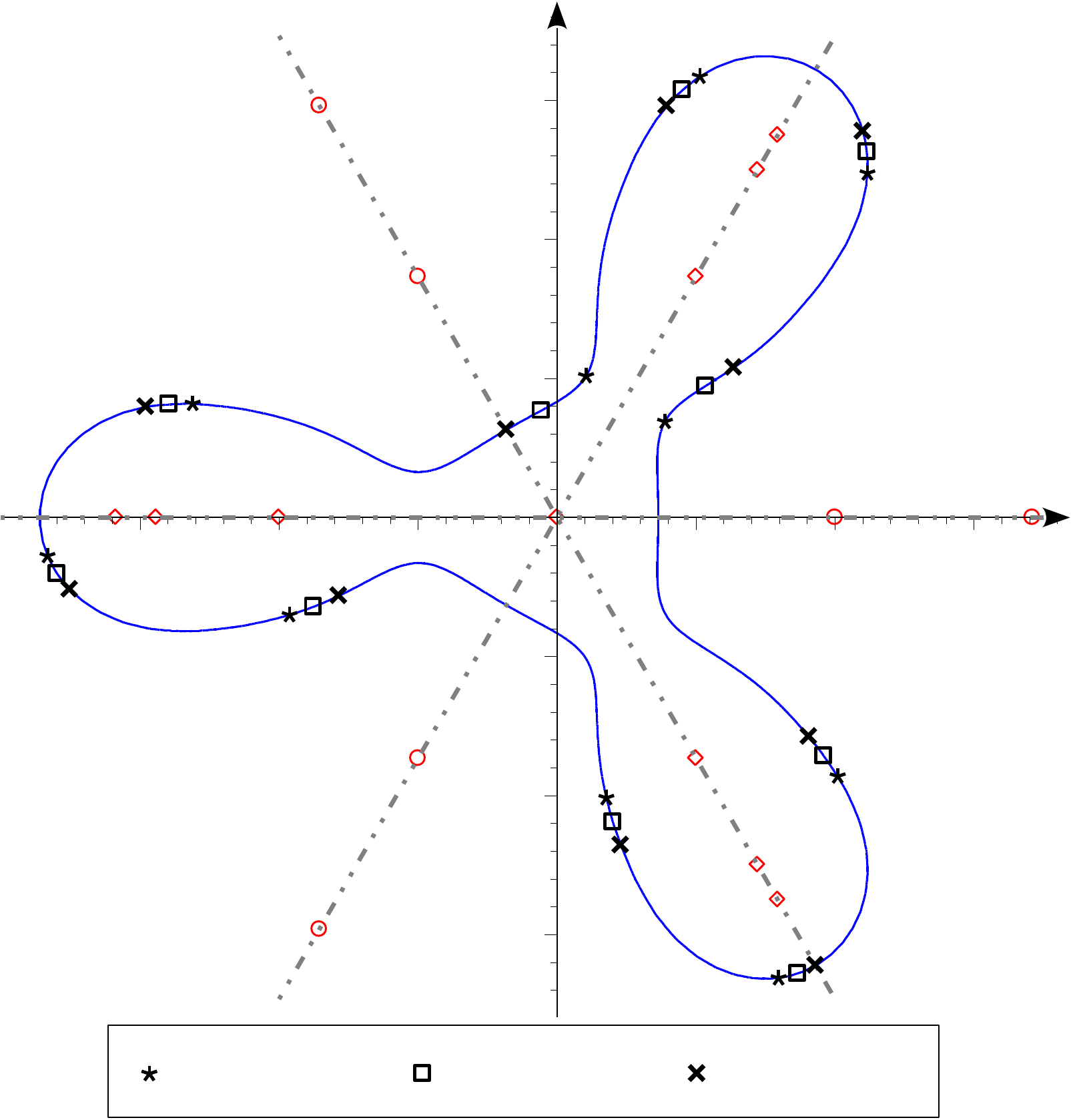}
                \put(155,30){$\alpha=e^{i\pi/3}$}
                \put(400,30){$\alpha=e^{i\pi /2}$}
                \put(645,30){$\alpha=e^{i 2\pi /3}$}

                \put(432,274){$-1$}
                \put(461,495){$0$}
                \put(461,773){$1$}
                \put(217,495){$-1$}
                \put(737,495){$1$}

                \put(859,485){$\Re z$}
                \put(390,950){$\Im z$}

                \green
                \put(820,685){$Q_0$}
                \put(810,365){$Q_{-1}$}
                \put(380,840){$Q_1$}
                \put(370,155){$Q_{-2}$}
                \put(150,755){$Q_2$}
                \put(140,300){$Q_{-3}$}
                
            \end{overpic}
            \vspace{0.5mm}
            \caption{\quad$F(z)=\dfrac{z(z^3+1)(z^3+3)(z^3+4)}{(z^3-1)(z^3-5)}$}
            \label{fig:0a}
        \end{subfigure}
        \begin{subfigure}{.495\linewidth}
            \centering
            \begin{overpic}[width=0.98\columnwidth%,grid,tics=25
                ]{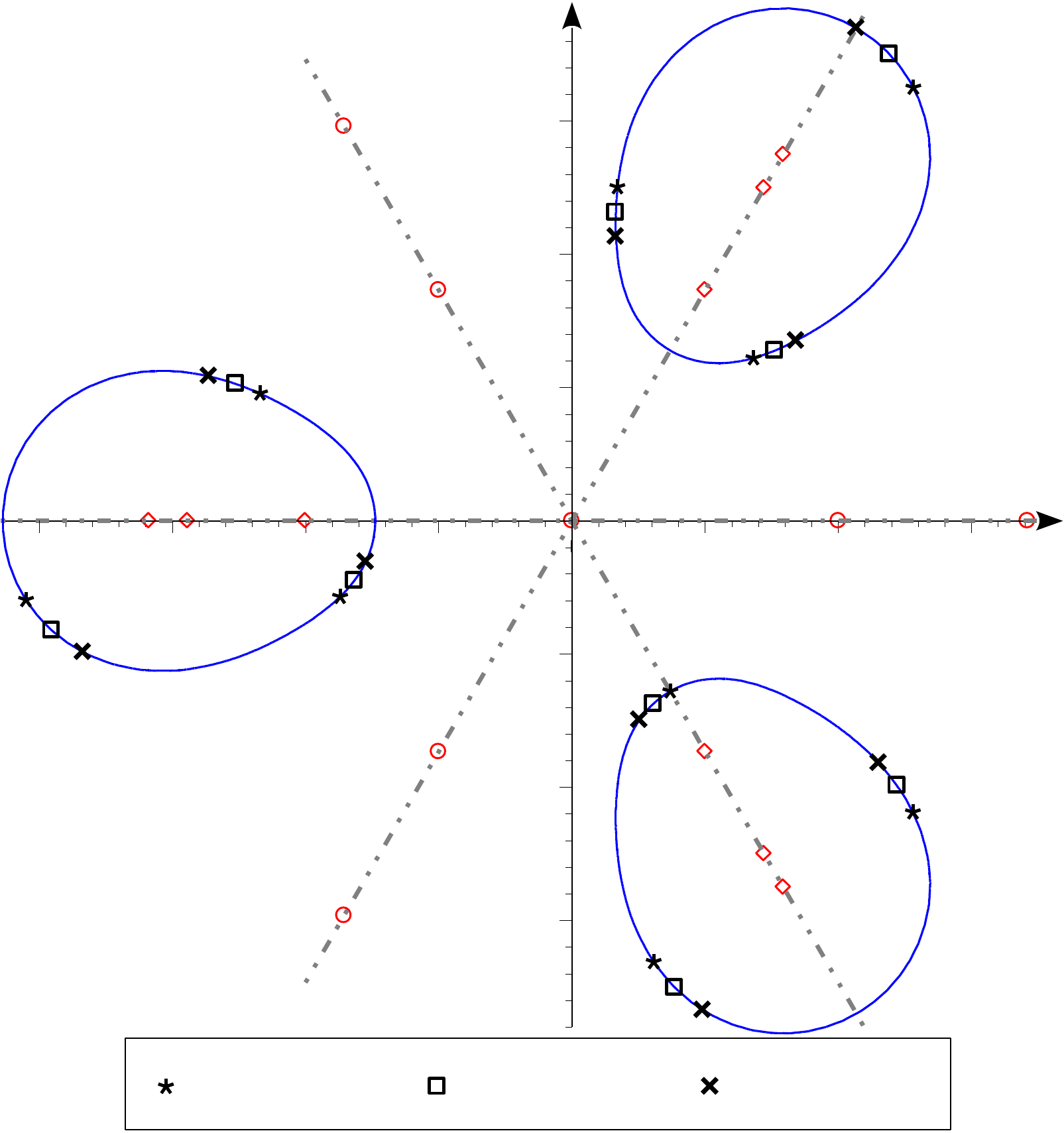}
                \put(164,30){$\alpha=e^{i\pi/3}$}
                \put(406,30){$\alpha=e^{i\pi /2}$}
                \put(648,30){$\alpha=e^{i 2\pi /3}$}
                % \put(255,30){$\alpha=e^{i\pi/3}$}
                % \put(427,30){$\alpha=e^{i\pi /2}$}
                % \put(599,30){$\alpha=e^{i 2\pi /3}$}

                \put(441,288){$-1$}
                \put(470,498){$0$}
                \put(470,760){$1$}
                \put(238,498){$-1$}
                \put(733,498){$1$}

                \put(845,488){$\Re z$}
                \put(399,945){$\Im z$}
                
                \green
                \put(820,645){$Q_0$}
                \put(810,395){$Q_{-1}$}
                \put(380,840){$Q_1$}
                \put(370,155){$Q_{-2}$}
                \put(150,755){$Q_2$}
                \put(140,300){$Q_{-3}$}
                
            \end{overpic}
            \caption{\quad$F(z)=\dfrac{(z^3+1)(z^3+3)(z^3+4)}{z(z^3-1)(z^3-5)}$}
            \label{fig:0b}
        \end{subfigure}
        \caption{The solutions to $F(z)=\alpha$ with~$k=3$, \ $p=\pm1$ for different values
            of~$\alpha$.
            (The isoline~$|F(z)|=1$ and zeros of the numerator and denominator of~$F$ have the same
            marks as in Figure~\ref{fig:00}.)}
        \label{fig:0}
    \end{figure}

    \newlength{\myparindent}\setlength{\myparindent}{\parindent}
    \noindent
    \begin{minipage}{1.0\textwidth}
        \begin{wrapfigure}{o}{.4\linewidth}
            \centering
            \begin{overpic}[width=\linewidth%,grid,tics=25
                ]{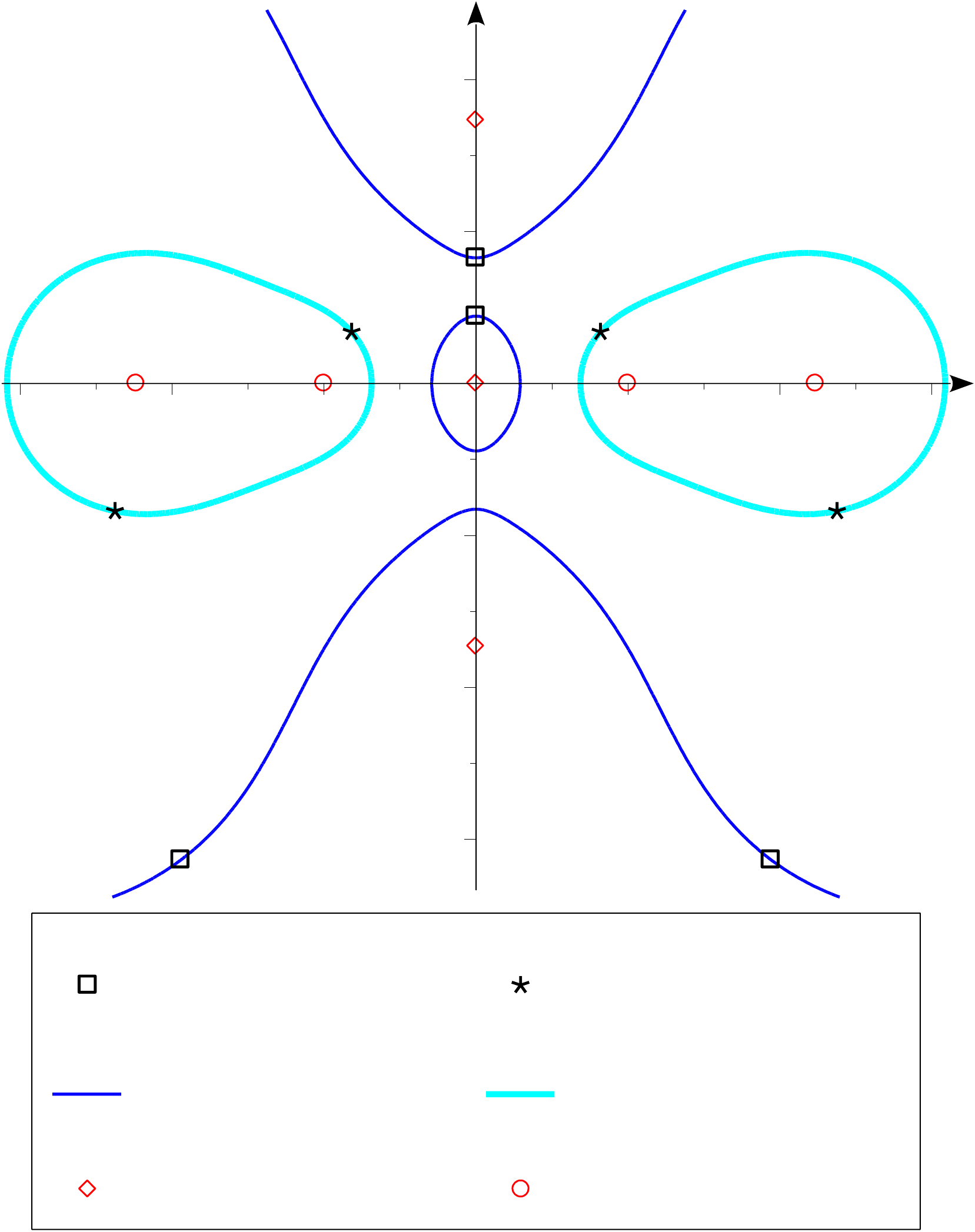}
                
                \put(110,189){$\frac{zf(z^2)}{g(z^2)}=\alpha=i$}
                \put(463,189){$\frac{zf(z^2)}{g(z^2)}=\alpha=\frac i5$}
                \put(110,100){$\left|\frac{zf(z^2)}{g(z^2)}\right|=1$}
                \put(463,100){$\left|\frac{zf(z^2)}{g(z^2)}\right|=\frac 15$}
                \put(110,20){$zf(z^2)=0$}
                \put(463,20){$g(z^2)=0$}

                \put(318,428){$-2$}
                \put(318,303){$-3$}
                \put(362,651){$0$}
                \put(350,920){$2$}
                \put(108,646){$-2$}
                \put(623,646){$2$}

                \put(673,721){$\Re z$}
                \put(410,950){$\Im z$}
            \end{overpic}
            \caption{The $\alpha$-points corresponding to}\label{fig:00}
            $f(z)=z+3$ and $g(z)=(z-1)(z-5)$.
        \end{wrapfigure}
        %\vspace{.2em}
        \setlength{\parindent}{\myparindent}%
        \setstretch{1.25} Note that % if needed???
        in the case when~$\alpha\in Q_{2q-1}$, the
        increment~\(\arg\alpha -\arg|\alpha| = \Arg(\alpha e_{-2q})\) is negative. At the same
        time~$\overline\alpha\in Q_{-2q}$, thus~$\overline\alpha e_{4q}\in Q_{2q}$ is the point
        of~$\Omega$ inducing the positive increment less then~$\frac{\pi}k$. % or <
        By Theorem~\ref{th:W_alpha_pts_order},~$R(w)\notin\Omega$ on condition that
        $|w_*|<|w|<|w_0|$, and
        moreover,~$R(w_0)=\overline\alpha{}e_{4q}e_{-2q}=\overline\alpha{}e_{2pm}$.
        Consequently,~$z_0=\sqrt[k]{\overline w_0}e_{2m}\in Q_{2m-1}$ as stated in
        Remark~\ref{rem:stretch-sect-nth-power} (for the illustration see Figure~\ref{fig:0a}
        with~$\alpha=e^{i\pi/2}$). Combining two last cases gives the
        implication~$\alpha\in Q_{2q-\varkappa}\implies z_0\in Q_{2m-\varkappa}$ and the
        simplicity of~$z_0$ by Theorem~\ref{th:main}.

        If~$\alpha e_{-2q+1}>0$,
        then~$\overline\alpha e_{4q} -\arg R(w_*)=\Arg e_{-1}=\frac{\pi}{k}$. Analogously to the
        previous case, $R(w_0)=\overline\alpha{}e_{2pm}$,
        therefore the equality~$z_0=\sqrt[k]{\overline w_0}e_{2m}\in\widetilde Q_{2m-1}$ determines
        the $\alpha$-point with the smallest absolute value. The situation
        when~$z_0\in Q_{2m-1}$ appears in Figure~\ref{fig:00}, $\alpha=\frac i5$, and
        Figure~\ref{fig:0a}, $\alpha=e^{i\pi/3}$. Theorem~\ref{th:main2} yields that there
        exists one another~$\alpha$-point of~$F$ with the same absolute value,
        namely~$z_0^* \colonequals \overline z_0e_{-2s}\in Q_{-2m-2s}$.

        The case of~$z_0 e_{-2m+1}>0$, that is~$w_0<0$, needs additional attention. The
        interval~$(-a_1, 0)$ contains one double (namely~$w_0$) or two simple ($w_0$ and~$w_1$)
        solutions to~$R(w_0)\in\Omega$ as provided by~\eqref{item:b_lemma2xx} of
        Lemma~\ref{lemma:phi_is_real_XX} with~$\psi=\ln R$. In addition,~$R(w)\notin\Omega$ on
        condition that~$|w_0|<|w|<|w_1|$, which is given by the
        inequality~\eqref{eq:re_psi_ineq}. These solutions determine the corresponding
        properties of the double $\alpha$-point~$z_0$ or, respectively, of the simple
        pair~$z_0$, $z_1$ with~$z_1e_{-2m+1}>0$ (as it is shown in Figure~\ref{fig:00}
        for~$\alpha=i$).
    \end{minipage}\vspace{1.0ex} % ????: if needed

    \setstretch{1.2}%
    In addition, the case when~$z_0 e_{-2m+1}>0$ is possible only if~$p=1$. Indeed, let~$w_0<0$
    and~$p\ge 2$. There exist a piecewise-smooth curve~$\gamma_1\subset\overline\Cp$
    connecting~$w_*$ and~$w_0$, such that~$\psi$ is holomorphic in its neighbourhood (see
    Lemma~\ref{lemma:all_vals}). Consider the closed contour
    $\gamma_2=\overline\gamma_1\cup\{w\in\mathbb C:\overline w\in\gamma_1\}$ and the enclosed
    domain~$D$ with the boundary~$\gamma_2$. The point~$w_*$ lies between the origin and the
    minimal pole, therefore the domain~$D$ contains no poles of~$R$. Consequently, there are no
    poles of the function~$S(w)\colonequals \sqrt[k]{w^{-p}}R(w)$ in~$D$ as well. Since~$S(w)$
    has the form~\eqref{eq:V_is_ln_psi} with~$B=A_0=0$, it is holomorphic in~$D$ and its
    argument on~$\gamma_2$ has the increment~$2 l\pi$ with some nonnegative integer~$l$ (we
    assume that~$\gamma_1$ and~$\gamma_2$ are directed anticlockwise). The increment
    on~$\gamma_1$ then must be~$l\pi\ge 0$, because~$S(\overline w)=\overline{S(w)}$. At the
    same time, $\arg\sqrt[k]{w^{p}} = \arg \frac{R(w)}{S(w)}$ has the increment~$\frac{p}{k}\pi$
    on~$\gamma_1$, which
    implies~$\arg R(w_0)-\arg R(w_*)=\left(\frac{p}k + l\right)\pi\ge\frac{p}k \pi$. Therefore,
    the condition~$\arg R(w_0)-\arg R(w_*)=\frac{\pi}k$, which is required for~$w_0<0$, fails to
    hold unless~$p=1$.
\end{proof}
Observe that the change of variable~$z\mapsto \zeta e_{-1}$ implies~$z^k\mapsto -\zeta^k$.
Hence, the function
\[
\widetilde F(\zeta) \colonequals \frac{e_{-p}}{F(\zeta e_{-1})}
= \zeta^{-p} e^{A\zeta^k}\,
{\prod_{\mu=1}^{\omega_2} \left(1+\frac{\zeta^k}{b_\mu}\right)}\ \Bigg/\ 
{\prod_{\nu=1}^{\omega_1} \left(1-\frac{\zeta^k}{a_\nu}\right)}
\]
has the form~\eqref{eq:F_mth} with a positive power of~$\zeta$ as the first factor provided
that~$p<0$. Moreover,
\begin{equation}\label{eq:z_zeta_corr}
    \begin{gathered}
        F(z) = \alpha \iff \widetilde F(\zeta) = \frac{e_{-p}}{\alpha}\equalscolon \widetilde\alpha,\\
        \alpha\in Q_{2q-\varkappa} \iff \widetilde\alpha\in Q_{-2q+\varkappa-p-1}
        \quad\text{and}\quad
        \alpha e_{-2q+\varkappa}>0 \iff \widetilde\alpha e_{2q-\varkappa+p}>0.
    \end{gathered}
\end{equation}
This way the case of~$p<0$ can be reduced to the situation studied in the previous theorem.
Unfortunately, the notation convenient in Theorem~\ref{th:meromorphic_pos_p} suits this case
worse inducing more complicated relations.
\begin{theorem}\label{th:meromorphic_neg_p}
    Suppose that all conditions of Theorem~\ref{th:meromorphic_pos_p} hold excepting
    that~$p<0$.

    If~$\alpha\in Q_{2q-\varkappa}$, then the $\alpha$-point~$z_0$ of~$F(z)$ closest to the
    origin is simple and distinct in argument and absolute value from the succeeding
    $\alpha$-point. Furthermore, $z_0\in Q_{2m-\sigma}$ where $\sigma\colonequals\varkappa$ for
    even~$p$, \ $\sigma\colonequals 1-\varkappa$ for odd~$p$, and the integer~$m$ satisfies%
    \footnote{Recall that $\left\lceil\frac{p}{2}\right\rceil$ stands for the minimal integer
        greater then or equal to~$\frac{p}{2}$. Here
        $\left|\left\lceil\frac{p}{2}\right\rceil\right|\le \left|\frac{p}{2}\right|$ since
        $p<0$.} \(pm\equiv q -(-1)^\sigma\left\lceil\frac{p}{2}\right\rceil \pmod{k}\).

    If $\alpha e_{-2q+\varkappa}>0$, where $p$ and $\varkappa$ have the same parity, then the
    $\alpha$-point~$z_0$ of~$F(z)$ closest to the origin is simple and distinct in argument and
    absolute value from the succeeding $\alpha$-point. Moreover, $z_0e_{-2m+1}>0$ for
    \(p m\equiv q + \left\lceil\frac{p-1}2\right\rceil \pmod{k}\).

    If $\alpha e_{-2q+\varkappa}>0$, where $p$ and $\varkappa$ have distinct parity, then the
    two zeros of~$F(z)-\alpha$ closest to the origin (counting double zeros as two) are equal in
    absolute value or in argument. In the latter case, which it possible only when~$p=-1$, both
    the zeros belong to the ray $\{ze_{-2m}>0\}$. In the former case, one of them belongs
    to~$Q_{2m}$ and another belongs to~$Q_{-2s-2m-1}$. Here~$m$ solves
    $p m\equiv q-\left\lceil\frac{p+1}2\right\rceil \pmod{k}$ and~$s$ is as in
    Remark~\ref{rem:2_cases}.
\end{theorem}
\begin{proof}
    % and $\alpha\in Q_{2q-\varkappa} \iff \widetilde\alpha\in Q_{-2q+\varkappa-p-1}$.
    % The inequality $\alpha e_{-2q+\varkappa}>0$ corresponds to
    % $\widetilde\alpha e_{2q-\varkappa+p}>0$.
    With the notation
    \begin{equation}\label{eq:tilde_q_tilde_kappa}
        \widetilde q \colonequals -q + \left\lceil\frac{\varkappa-p-1}2\right\rceil
        \quad\text{and}\quad
        \widetilde\varkappa \colonequals 2\widetilde q +2q-\varkappa+p+1,
    \end{equation}
    the relations~\eqref{eq:z_zeta_corr} immediately yield
    \[
      \alpha\in Q_{2q-\varkappa} \iff \widetilde\alpha\in Q_{2\widetilde q-\widetilde\varkappa}
      \xLongrightarrow{\text{Theorem~\ref{th:meromorphic_pos_p}}}
      \zeta_0\in Q_{2\widetilde m-\widetilde\varkappa}\iff z_0\in Q_{2\widetilde m-\widetilde\varkappa-1},
    \]
    where $\widetilde m$ satisfies~$(-p)\cdot\widetilde m\equiv \widetilde q\pmod{k}$
    and~$\zeta_0$ is the solution to~$\widetilde F(\zeta)=\widetilde\alpha$ minimal in absolute
    value. That is, modulo~$k$ we have
    \begin{equation}\label{eq:tilde_m_via_q}
        p\widetilde m
        \equiv q - \left\lceil\dfrac{\varkappa-p-1}2\right\rceil
        =
        \begin{cases}
            q + \frac p2,& \text{if $p$ is even}\\
            q + \frac {p+1}2 -\varkappa,& \text{if $p$ is odd}
        \end{cases}
    \end{equation}

    \setstretch{1.12}
    Let~$m$ denotes such an integer that~$z_0\in Q_{2m-\sigma}$ for some~$\sigma\in\{0,1\}$.
    Then
    necessarily~$\relpenalty 100 2m-\sigma \equiv 2\widetilde m-\widetilde \varkappa-1
    \pmod{2k}$, which is satisfied by $m=\widetilde m-\widetilde \varkappa$
    and~$\sigma=1-\widetilde \varkappa$. At that, the second of the
    expressions~\eqref{eq:tilde_q_tilde_kappa} yields \(\widetilde\varkappa = 1-\varkappa\) if
    $p$ is even and $\widetilde\varkappa = \varkappa$ if $p$ is odd. The
    relation~\eqref{eq:tilde_m_via_q} within these settings becomes
    \[%begin{equation*}%\label{eq:tilde_m_via_q}
        pm
        % = q - \left\lceil\dfrac{\varkappa-p-1}2\right\rceil -p\varkappa
        % = q - \left\lceil\dfrac{\varkappa-1-(-1)^\varkappa p}2\right\rceil
        \equiv
        \begin{cases}
            q + \frac p2 - p\widetilde\varkappa,& \text{if $p$ is even;}\\
            q + \frac {p+1}2 -\widetilde\varkappa(p+1),& \text{if $p$ is odd}
        \end{cases}
        =
        \begin{cases}
            q + (-1)^{\widetilde\varkappa}\frac p2 ,& \text{if $p$ is even;}\\
            q + (-1)^{\widetilde\varkappa}\frac {p+1}2,& \text{if $p$ is odd}
        \end{cases}
    \]%end{equation*}
    modulo~$k$. Since~$p<0$, the last equality implies
    \[
    pm
    \equiv q + (-1)^{\widetilde\varkappa}\left\lceil\frac {p}2\right\rceil
    = q - (-1)^{\sigma}\left\lceil\frac {p}2\right\rceil\pmod{k}.
    \]
    However, this relation coincides with the relation for~$m$ suggested by the theorem's
    statement. For the corresponding illustration see Figure~\ref{fig:0b}, $\alpha=e^{i\pi/2}$.
    % \vspace{3ex}

    When~$\widetilde\alpha$ satisfies~$\widetilde\alpha e_{-2\widetilde q}>0$, from~\eqref{eq:z_zeta_corr}
    we have the relation~$-2\widetilde q \equiv 2 q-\varkappa+p \pmod{2k}$ instead
    of~\eqref{eq:tilde_q_tilde_kappa}, which determines the pair~$q,\varkappa$ satisfying the
    inequality~$\alpha e_{-2q+\varkappa}>0$. So, $z_0e_{-2m+1}>0$ for
    \[
    p m = -(-p)\cdot m \equiv -\widetilde q \equiv q+\frac{p-\varkappa}2
    = q + \left\lceil\frac{p-1}2\right\rceil
    \pmod{k}
    \]
    on account of the equal parity of~$p$ and~$\varkappa$. The corresponding plot can be found
    in Figure~\ref{fig:0b}, $\alpha=e^{i\pi/3}$.
    
    When~$\widetilde\alpha e_{-2\widetilde q+1}>0$, the
    relation~$-2\widetilde q+1 \equiv 2 q-\varkappa+p \pmod{2k}$ provides another pair $q,\varkappa$
    making the inequality~$\alpha e_{-2q+\varkappa}>0$ true. This gives us that
    $z_0\in Q_{2\widetilde m-2}$ or~$z_0e_{-2\widetilde m+2}>0$ (the latter is possible only for~$p=-1$)
    whenever
    \[p\widetilde m\equiv -\widetilde q \equiv q+ \frac{p-1-\varkappa}2 \pmod{k}.\]
    The change~$m\colonequals\widetilde m-1$ gives $z_0\in Q_{2m}$ or~$z_0e_{-2m}>0$ whenever
    \[
    p m\equiv q-p+\frac{p-1-\varkappa}2
    =q-\frac{p+1+\varkappa}2
    =q-\left\lceil\frac{p+1}2\right\rceil
    \pmod{k}.
    \]
    For~$z_0\in Q_{2m}$, the integer~$s$ defined as in Remark~\ref{rem:2_cases} provides
    the expression \( \overline z_0 e_{-2s} \) for the $\alpha$-point of~$F(z)$ which is
    equidistant with~$z_0$ from the origin. See the relevant example in Figure~\ref{fig:0b},
    $\alpha=e^{i2\pi/3}$.
\end{proof}

\begin{remark}
    Note that the last two theorems can be applied to the function~$H(1/z)$ when it has the
    form~\eqref{eq:F_mth}. This way one obtains a straightforward conclusion concerning
    {\myem the most distant} from~$0$ solution of the equation~$H(z)=\alpha$. It is of special
    interest for rational~$H(z)$: then both $H(z)$ and $H(1/z)$ can be represented as
    in~\eqref{eq:F_mth}.
\end{remark}

\section{Zeros of entire functions}
\label{sec:entire-functions}
Let the natural numbers~$j$ and $k$ be coprime and~$k\ge2$. Theorems~\ref{th:main},
\ref{th:main2}--\ref{th:meromorphic_neg_p} admit a transition to describing the zeros of
functions of the forms
\begin{equation*}%\label{eq:ent_funct_form2}
    H_1(z) \colonequals f(z^k) + z^{j} g(z^k)
    \quad\text{and}\quad
    H_2(z) \colonequals g(z^k) + z^{j} f(z^k),
\end{equation*}
where the functions~$f(z)$ and~$g(-z)$ are entire, have genus~$0$ and
{\myem only negative zeros}. At least one of the functions~$f$ and~$g$ needs to be not a
constant to exclude the trivial case. Furthermore, both functions $f(z^k)/f(0)$ and
$g(z^k)/g(0)$ must be real. They have no common zeros, therefore $f(z^k)\ne 0$, \ $g(z^k)\ne 0$
and~$z\ne 0$ when $H_1(z)=0$ or~$H_2(z)=0$.

To adapt the facts stated in Sections~\ref{sec:main-theorems}
and~\ref{sec:locat-alpha-points} for studying zeros of the functions~$H_1$ and~$H_2$, put
\begin{equation}\label{eq:Fi_repr}
    F_1(z)\colonequals z^{-j}\frac{f(z^k)/f(0)}{g(z^k)/g(0)}
    ,\qquad
    F_2(z)\colonequals z^{j}\frac{f(z^k)/f(0)}{g(z^k)/g(0)}
    \an
    \alpha\colonequals-\frac{g(0)}{f(0)}.
\end{equation}
Then the following identities hold:
\begin{equation}\label{eq:Hi_and_Fi}
    H_1(z) = \left(1-\frac{F_1(z)}\alpha\right)z^jg(z^k)
    \an
    H_2(z) = \left(1-\frac{F_2(z)}\alpha\right)g(z^k).
\end{equation}
Recall that~$z^jg(z^k)$ and~$H_i(z)$ have no common zeros for~$i=1,2$. Therefore, the
equalities~\eqref{eq:Hi_and_Fi} imply that $F_1(z)=\alpha\iff H_1(z)=0$ and
$F_2(z)=\alpha\iff H_2(z)=0$. That is, the zero set of~$H_i(z)$ {\myem coincides} with
the~$\alpha$-set of~$F_i(z)$ for~$i=1,2$. Moreover,~\eqref{eq:Hi_and_Fi} give that each
$\alpha$-point~$z_*$ of the function~$F_i(z)$ is the zero of~$H_i(z)$ with
{\myem the same multiplicity}.

Since the functions~$F_i(z)$ have the form~\eqref{eq:F_mth}, the zeros of~$H_i(z)$ for~$i=1,2$
are located as it is asserted about $\alpha$-points of~$F_i(z)$ by Theorems~\ref{th:main},
\ref{th:main2}--\ref{th:meromorphic_neg_p} with~$\alpha=-\frac{g(0)}{f(0)}$
and~$p=(-1)^i\cdot j$.

\begin{remark}\label{rem:zeros-entire-funct-extensions}
    Some extensions of the fact proposed in the current section are possible. Here we give two
    examples. However, it is unclear whether studying such functions is well-motivated.
    \begin{asparaenum}[1.]
    \item Assume that~$f(z)$ and~$g(z)$ are functions regular and nonzero at the origin, and
        that $\frac{f(z)}{g(-z)}$ does not coincide with~$z^p$ up to a constant (to suppress the
        trivial case). From the comparison of the formulae~\eqref{eq:Fi_repr}
        with~\eqref{eq:G_mth} and~\eqref{eq:F_mth} it is seen that~$f(z)/f(0)$ and~$g(-z)/g(0)$
        can be allowed to have the form
        \begin{equation}\label{eq:funct_gen_stps}
            e^{Az}\cdot {\prod_{\nu>0} \left(1+\frac{z}{a_\nu}\right)}\ \bigg/\
            {\prod_{\mu>0} \left(1-\frac{z}{b_\mu}\right)},\quad\text{where} \quad
            A\ge 0\an a_\nu,b_\mu>0\text{ for all $\mu$, $\nu$}.
        \end{equation}
        Put in other words, if~$f(0),g(0)\in\mathbb{C}\setminus\{0\}$ then $f(z)/f(0)$
        and~$g(-z)/g(0)$ can generate any \emph{totally positive sequences} which start
        with~$1$. Indeed, after multiplying~$H_i$ by the denominators of~$f(z^k)$ and~$g(z^k)$
        we obtain the entire function~$\widetilde H_i$ with the same zeros as~$H_i$. Then it is
        enough to note that the exponential factor originating from those of~$f(z^k)$
        and~$g(z^k)$ is allowed in the representation~\eqref{eq:F_mth}. So, the result of the
        current section extends to such functions without any changes.
    \item Let $f(z)$ and~$g(-z)$ be nontrivial functions generating
        \emph{doubly infinite totally positive sequences} up to a complex constant factor,
        \textit{i.e.} of the form~\eqref{eq:funct_gen_dtps} with a complex~$C\ne 0$. In addition
        let $f(z)\not\equiv~\text{const}\cdot z^pg(-z)$. With the help of analogous
        manipulations we still can obtain a transition of Theorems~\ref{th:main}
        and~\ref{th:main2}. After taking out the factors~$z^p$ of~\eqref{eq:funct_gen_dtps}, it
        is enough additionally to factor some power of~$z$ out of~$H_i$ (this cannot change the
        zero set excepting the origin).
    \end{asparaenum}
\end{remark}
\begin{remark}\label{rem:zeros-entire-funct-extensions-plus}
    Allowing $f(z)$ and~$g(-z)$ to be arbitrary functions of the forms~\eqref{eq:funct_gen_stps}
    or~\eqref{eq:funct_gen_dtps} with~$C\in\mathbb{C}\setminus\{0\}$ can be useful in the
    following sense. Consider the power series
    \[
    f(z)=\sum_{n=-\infty}^\infty f_nz^n
    \an
    g(z)=\sum_{n=-\infty}^\infty g_nz^n
    \]
    such that $f_n\ne f_0^{1-n}f_1^n$ and~$g_m\ne g_0^{1-m}g_1^{m}$ for some~$n,m\in\mathbb{Z}$.
    Then (see the discussion on Page~\pageref{eq:funct_gen_dtps}) the series converge and the
    functions~$f(z)/f_0$ and~$g(-z)/g_0$ \emph{generate totally positive sequences} (possibly
    doubly infinite) if and only if the Toeplitz matrices
    \(\left(f_{n-i}/f_0\right)_{i,n=-\infty}^\infty\) and
    \(\left((-1)^{n-i}g_{n-i}/g_0\right)_{i,n=-\infty}^\infty\) have all their minors
    nonnegative~\cite{AESW,Edrei}. However, then
    \[
    H_1(z)=\sum_{n=-\infty}^\infty \big(f_n+z^jg_n\big) z^{kn}
    \an
    H_2(z)=\sum_{n=-\infty}^\infty \big(g_n+z^jf_n\big) z^{kn}
    \]
    are the Laurent series. This gives us the conditions in terms of
    \emph{the Laurent coefficients} of~$H_1(z)$ and~$H_2(z)$ which provide that the
    zeros~of~$H_1(z)$ and~$H_2(z)$ are localized according to Theorems~\ref{th:main}
    and~\ref{th:main2} (and to Theorems~\ref{th:meromorphic_pos_p}--\ref{th:meromorphic_neg_p}
    when the series~$f(z)$, $g(z)$ are not doubly infinite, that the limiting functions is are
    meromorphic).
\end{remark}
\setstretch{1.15}%
\section{Conclusions for the case~\texorpdfstring{$k=2$}{k=2}}
\label{sec:concl_for_k_2}
Note that in the particular case of $p=\pm 1$ the relations modulo~$k$ from
Theorems~\ref{th:main},~\ref{th:main2}--\ref{th:meromorphic_neg_p} have obvious solutions. The
% additional 
setting $k=2$ (implying that~$p$ is odd) also provides us with simple (and very useful)
solutions. Let us restate the facts of the previous section for this particular situation.

Denote~$p=2j+1$. The congruence modulo~$k$ (a linear Diophantine equation) from
Theorem~\ref{th:main} becomes~$l\equiv 1+\varkappa+\sigma+m\pmod{2}$. If~$\alpha\in\mathbb{R}$
(or~$i\alpha\in\mathbb{R}$), then the constant~$s$ in
Theorems~\ref{th:main2}--\ref{th:meromorphic_neg_p} equals~$0$ (or~$1$, respectively). The
congruence from Theorem~\ref{th:main2} turns into~$l=1+m\pmod{2}$. The equation from
Theorem~\ref{th:meromorphic_pos_p} becomes~$m=q\pmod{2}$, and those from
Theorem~\ref{th:meromorphic_neg_p} become
\[
m\equiv
\begin{cases}
    q + j+1 \pmod{2},&\text{if}\quad \alpha\in Q_{2q-\varkappa}\quad\text{or}\quad\alpha e_{-2q}>0;\\
    q + j \pmod{2},&\text{if}\quad \alpha e_{-2q+1}>0.
\end{cases}
\]
Let $N=(z_k)_{k=1}^\omega$ be the set of all $\alpha$-points of~$F(z)$, where
$|z_{k-1}|\le |z_k|$ for all~$k$ and each $\alpha$-point counts so many times which multiplicity
it has. Then we have the following two theorems as a summary.

\begin{theorem}[\textit{cf.}~\cite{Dyachenko1}]\label{th:ca1}
    Let a function~$F(z)$ have the form~\eqref{eq:F_mth}, $p=2j+1$, \ $j<0$ and $k=2$. then the
    $\alpha$-points~$N=(z_k)_{k=1}^\omega$ of~$F(z)$ are distributed as follows:
    \begin{enumerate}[\upshape(i)]
    \item\label{distr1} If\/~$\Im \alpha^2\ne 0$, then all zeros are simple and satisfying
        $0<|z_1|<|z_2|<\dots$; $\Im z_n^2\ne 0$~for every natural~$n$
        and~$z_n\in Q_l\implies z_{n+1}\in Q_{l+\sign\left(\Im \alpha^2\right)}$.
        Moreover,
        %$(-1)^j\Re \alpha\Re z_1<0$
        $(-1)^j\Im \alpha\Im z_1 > 0$
        and $\Im \alpha^2\Re z_1\Im z_1<0$.
    % \item\label{distr1a} If\/~$\Im \alpha^2<0$, then all zeros are simple and satisfying
    %     $0<|z_1|<|z_2|<\dots$; $\Im z_n^2\ne 0$~for every natural~$n$
    %     and~$z_n\in Q_l \implies z_{n+1}\in Q_{l-1}$.
    %     Moreover,
    %     %$(-1)^j\Re \alpha\Re z_1<0$
    %     $(-1)^j\Im \alpha\Im z_1 > 0$
    %     and $\Re z_1\Im z_1>0$.
    \item\label{distr2} If\/~$\Im \alpha=0$, then $0<|z_1|\le|z_2|<|z_3|\le|z_4|<|z_5|\dots$, there are
        no purely imaginary zeros, real zeros can be simple or double, other zeros are simple.
        Moreover, for each natural $n$ the following five conditions hold
        \begin{gather*}
            |z_{2n-1}|=|z_{2n}| \implies z_{2n-1}=\overline{z}_{2n},\qquad\qquad
            |z_{2n-1}|<|z_{2n}| \implies \Arg z_{2n-1}=\Arg{z}_{2n}\in\{0,\pi\},\\
            \Re z_{2n}\Re z_{2n+1}<0,\qquad\qquad
            (-1)^j\alpha\Re z_1 < 0\qquad\qquad\text{and}\qquad\qquad
            |z_{1}|<|z_{2}| \implies j=-1.
        \end{gather*}
    \item\label{distr3} If\/~$\Re \alpha=0$, then $0<|z_1|<|z_2|\le|z_3|<|z_4|\le|z_5|\dots$,
        there are no real zeros, purely imaginary zeros can be simple or double, other zeros are
        simple. Moreover, for each natural $n$ the following five conditions hold
        \begin{gather*}
            |z_{2n}|=|z_{2n+1}| \implies z_{2n}=-\overline{z}_{2n+1},\qquad\qquad
            |z_{2n}|<|z_{2n+1}| \implies
                    \Arg z_{2n}=\Arg{z}_{2n+1}\in\left\{-\frac\pi2,\frac\pi2\right\},\\
            \Im z_{2n-1}\Im z_{2n}<0,\qquad\qquad
            (-1)^j\Im \alpha\Im z_1 > 0 \qquad\qquad\text{and}\qquad\qquad \Re z_1 = 0.
            %the last two conditions are equivalent to (-1)^j\alpha z_1 < 0
        \end{gather*}
    \end{enumerate}
\end{theorem}

\begin{theorem}\label{th:ca2}
    Let a function~$F(z)$ have the form~\eqref{eq:F_mth}, $p=2j+1$, \ $j\ge 0$ and $k=2$. Then
    the $\alpha$-points $N=(z_k)_{k=1}^\omega$ of~$F(z)$ are distributed as follows:
    \begin{enumerate}[\upshape(i)]\addtocounter{enumi}{3}
    \item\label{distr4} If\/~$\Im \alpha^2\ne 0$, then all zeros are simple and satisfying
        $0<|z_1|<|z_2|<\dots$; $\Im z_n^2\ne 0$~for every natural~$n$,
        $z_n\in Q_l\implies z_{n+1}\in Q_{l+\sign\left(\Im \alpha^2\right)}$. Moreover,
        $\Im\alpha\Im z_1 > 0$ and $\Re\alpha\Re z_1>0$.
    \item\label{distr5} If\/~$\Im \alpha=0$, then $0<|z_1|<|z_2|\le|z_3|<|z_4|\le|z_5|\dots$,
        there are no purely imaginary zeros, real zeros can be simple or double, other zeros are
        simple. Moreover, for each natural $n$ the following five conditions hold
        \begin{gather*}
            |z_{2n}|=|z_{2n+1}| \implies z_{2n}=\overline{z}_{2n+1},\qquad\qquad
            |z_{2n}|<|z_{2n+1}| \implies
                    \Arg z_{2n}=\Arg{z}_{2n+1}\in\{0,\pi\},\\
            \Re z_{2n-1}\Re z_{2n}<0,\qquad\qquad
            \Re z_1 = 0 \qquad\qquad\text{and}\qquad\qquad \alpha z_1 > 0.
        \end{gather*}
    \item\label{distr6} If\/~$\Re \alpha=0$, then $0<|z_1|\le|z_2|<|z_3|\le|z_4|<|z_5|\dots$,
        there are no real zeros, purely imaginary zeros can be simple or double, other zeros are
        simple. Moreover, for each natural $n$ the following five conditions hold
        \begin{gather*}
            |z_{2n-1}|=|z_{2n}| \implies z_{2n-1}=-\overline{z}_{2n},\qquad\qquad
            |z_{2n-1}|<|z_{2n}| \implies \Arg z_{2n-1}=\Arg{z}_{2n}\in\left\{-\frac\pi2,\frac\pi2\right\},\\
            \Im z_{2n}\Im z_{2n+1}<0,\qquad\qquad
            \Im\alpha\Im z_1 > 0\qquad\qquad\text{and}\qquad\qquad
            |z_{1}|<|z_{2}| \implies j=1.
        \end{gather*}
    \end{enumerate}
\end{theorem}

\begin{remark}\label{rem:akin_facts}
    Two last theorems have analogous statements if, instead of~$F(z)$
    satisfying~\eqref{eq:F_mth}, we take a function $G(z)$ of the form~\eqref{eq:G_mth}. Then,
    generally speaking, we cannot assert where the $\alpha$-point that is the least in absolute
    value occurs (it may be absent at all).
\end{remark}
\begin{remark}
Note that in all cases~\eqref{distr1}--\eqref{distr6} the $\alpha$-point split evenly among the
quadrants of complex plane. That is, if the $\alpha$-set of a function satisfies~\eqref{distr1}
or~\eqref{distr4}, then for any $r>0$ the number of $\alpha$-points in $\{z\in Q_k:|z|<r\}$ can
differ from the number of $\alpha$-points in $\{z\in Q_j:|z|<r\}$ at most by $1$ (here
$k,j=1,\dots,4$). The cases appearing in~\eqref{distr2}, \eqref{distr3}, \eqref{distr5}
and~\eqref{distr6} are the ``degenerated'' those of \eqref{distr1} and~\eqref{distr4} with
possible ingress of some $\alpha$-points onto the real or imaginary axes.
\end{remark}

Let us turn to zeros of entire functions by applying the idea of
Section~\ref{sec:entire-functions}. An entire function~$\relpenalty 100 H(z)=\sum_{n=0}^\infty f_{n} z^{n}$,
where $f_0\ne 0$, splits up into the even and odd parts according to
\begin{equation}\label{eq:H_even_odd}
    H(z)=f(z^2)+zg(-z^2),\quad\text{where}\quad
    f(z)=\sum_{n=0}^{\infty} f_{2n}z^n \an
    g(z)=\sum_{n=0}^{\infty} f_{2n+1}z^n.
\end{equation}
Since \(\displaystyle H(z)=0 \iff \frac{f(z^2)/f_0}{zg(z^2)/f_1} = -\frac {f_1}{f_0}\), zeros of
$H(z)$ are distributed as stated in Theorem~\ref{th:ca1} if both $f(z)/f_0$ and $g(z)/f_1$ have
only negative zeros and the genus equal to~$0$ up to factors of the form~$e^{cz}$, \ $c\ge 0$.
Similarly, zeros of $H(z)$ are distributed as stated in Theorem~\ref{th:ca2} provided that both
$f(z)/f_0$ and $g(z)/f_1$ have only positive zeros and the genus equal to~$0$ up to factors of
the form~$e^{-cz}$, \ $c\ge 0$.

A real entire function $\widetilde H(z) = f(z^2)+zg(z^2)$ is {\myem strongly stable} if
$f(z^2) + (1+\eta) z g(z^2)$ has no zeros in the closed right half of the complex plane for all
complex~$\eta$ which are small enough. This is the ``proper'' extension to entire functions of
the polynomial stability. The strongly stable functions include all stable polynomial and
functions which are close to them in some sense (for example, $e^x$); but they exclude such
functions like $e^{-x}$. For the case of complex functions and further details
see~\cite[p.~129]{Postnikov}. Note that~$\widetilde H(z)$ is strongly stable whenever
$\widetilde H(iz)$ belongs to the class~\textbf{HB} (named after Hermite and Biehler), which is
introduced in~\cite{ChebMei}.

\begin{remark}
    Chebotarev's Theorem (see~\cite{Chebotarev} or \emph{e.g.}~\cite{ChebMei,Postnikov}) implies
    that if~$\widetilde H(z)$ is strongly stable then $f(g)$ and~$g(z)$ have genus~$0$ up to
    factors~$e^{cz}$, $c>0$, and their zeros are negative, simple and interlacing. In
    particular,
    \emph{if $\widetilde H(z)$ is a strongly stable function, then zeros of the function~$H(z)$,
        which is defined by~\eqref{eq:H_even_odd}, are distributed as stated in
        Theorem~\ref{th:ca1}}. At that, the interlacing property of~$f(z)$ and~$g(z)$ remains
    redundant.
\end{remark}

Observe that if a complex number $\mu$ satisfies $\mu^4=-1$ (\textit{i.e.} $\mu$ is a
primitive $8$th root of unity), then we have the identity
\begin{multline*}
    % G(\overline\mu z)
    % \colonequals&
      \sum_{n=0}^\infty f_{n} \mu^{n(n-1)} (\mu^{-1} z)^{n}
      =\sum_{n=0}^\infty f_{n} \mu^{n(n-2)} z^{n}
      = \left(\sum_{n=0}^\infty f_{2n} \mu^{4n(n-1)} z^{2n}
          +\sum_{l=0}^\infty f_{2l+1} \mu^{4l^2-1} z^{2l+1}\right)\\%(2l+1)(2l+3)
      = \left(\sum_{n=0}^\infty f_{2n} 1^{\frac{n(n-1)}{2}} z^{2n}
          +\sum_{l=0}^\infty f_{2l+1} (-1)^{l^2}\mu^{-1} z^{2l+1}\right)
      =\sum_{n=0}^\infty f_{2n} z^{2n} +\overline\mu \sum_{l=0}^\infty (-1)^l f_{2l+1} z^{2l+1}.
 \end{multline*}
Consequently, the following fact is true.
\begin{corollary}\label{cr:ca2}
    Consider the functions
    \[
    h(z) = \sum_{k=0}^\infty i^{\frac{k(k-1)}2} f_{k} z^{k} \quad\text{and}\quad
    \overline{h}(z) = \sum_{k=0}^\infty i^{-\frac{k(k-1)}2} f_{k} z^{k},
    \]
    where $f(z)=\sum_{n=0}^{\infty} f_{2n}z^n$ and $g(z)=\sum_{n=0}^{\infty} f_{2n+1}z^n$ are
    entire functions of genus~$0$ and have only negative zeros. For $\mu=\sqrt{i}$, zeros of the
    function $h(\overline\mu z)$ distributed as claimed in Theorem~\ref{th:ca1} for
    $\alpha=\overline\mu f_1/f_0$ and zeros of the function $\overline h(\mu z)$ distributed as
    claimed in Theorem~\ref{th:ca1} for~$\alpha=\mu f_1/f_0$.
\end{corollary}

\section{Two problems by A. Sokal}\label{sec:two-problems-sokal}
\paragraph*{``Disturbed exponential'' function.\hspace{-.5em}}\setstretch{1.24}
Alan Sokal in his talk at Institut Henri Poincar\'e (see~\cite{Sokal}) put forward the
hypothesis that
\begin{conjecture}\label{conjecture-A}
    The entire function
    \begin{equation}\label{eq:0}
        \mathcal F(z;q)=\sum_{n=0}^\infty \frac {q^{\frac{n(n-1)}2}z^n}{n!},
    \end{equation} 
    where $q$ is a complex number, $0<|q|\leq 1$, can have only simple zeros.
\end{conjecture}
The stronger version of the conjecture claims that
\begin{conjecture}\label{conjecture-B}
    The function $\mathcal F(z;q)$ for $q\in\mathbb{C}$, \ $\ 0<|q|\leq 1$, can have only simple
    zeros with distinct absolute values.
\end{conjecture}

The following facts on~$\mathcal F(z;q)$ are known. The function $\mathcal F$ is the unique
solution to the following Cauchy problem
\begin{equation}\label{eq:1}
    \mathcal F'(z)=\mathcal F(qz),\quad \mathcal F(0)=1,
\end{equation}
which can be checked by substitution. Moreover, when $|q|=1$ this function has the exponential
type~$1$, for~$q$ lower in absolute value the function~$\mathcal F$ is of zero genus.

The case of positive~$q$ was studied extensively. It is known that the zeros of~$\mathcal F$ are
negative (see~\cite{Morris_et_al}), simple and satisfy Conjecture~\ref{conjecture-B} as well as
certain further conditions~\cite{Liu,Langley}. Conjecture~\ref{conjecture-B} holds true for
negative~$q$ as well, see~\emph{e.g.}~\cite{Dyachenko1}. The properties of~$\mathcal F(z;q)$ for
complex~$q$ were studied in~\cite{Alander,Valiron,Eremenko_Ostrovsky}. According
to~\cite{Sokal}, Conjecture~\ref{conjecture-B} is true if $|q|<1$ and the zeros
of~$\mathcal F(z;q)$ are big enough in absolute value (A.~Eremenko) as well as for small $|q|$.

Let us prove that Conjecture~\ref{conjecture-B} also holds true for purely imaginary values of
the parameter. As we pointed out, for positive~$q\le 1$ the
function~$\mathcal F(z;q)=f(z^2)+zg(z^2)$ has only negative zeros. In particular, it is stable.
The Hermite-Biehler theory (e.g.~\cite{ChebMei,Postnikov}) implies that the zeros of~$f(z)$
and~$g(z)$ are negative and interlacing. Therefore, by Corollary~\ref{cr:ca2} the zeros of
$\mathcal F(z;\pm iq)$ with~$0<q\le 1$ are simple and their absolute values are distinct.\qed

The family of polynomials
\begin{equation}\label{eq:01}
    P_N(z;q)=\sum_{n=0}^N \binom{N}{n}z^nq^{\frac{n(n-1)}{2}}
\end{equation} 
is connected tightly to the function $\mathcal F(z;q)$; it approximates this function in the sense that
\[
P_N\left(zN^{-1};q\right)
    = \sum_{n=0}^N \frac{q^{\frac{n(n-1)}{2}}z^n}{n!}
           \left(1-\frac{1}N\right)\cdot\left(1-\frac{2}N\right)
               \cdots\left(1-\frac{n-1}N\right) \xrightarrow{N\to\infty} \mathcal F(z;q).
\]
The polynomial version of this conjecture has the following form.
\begin{conjecture}\label{conjecture-C}
    For all $N>0$ the polynomial $P_N(z;q)$ where $|q|<1$, can have only simple roots, separated
    in absolute value by at least the factor $|q|$.
\end{conjecture}
\setstretch{1.2}
The original statement (which is equivalent to the given here) is concerned with the family of
polynomials $\left\{P_N\left(zw^{N-1};w^{-2}\right)\right\}_{N\in\mathbb{N}}$, where $w^{-2}=q$.
Observe that $\ref{conjecture-C}\implies \ref{conjecture-B}\implies \ref{conjecture-A}$.

The approach for~$\mathcal F(z;q)$ extends to the polynomials~$P_N(z;q)$ without changes. Their
zeros are negative for positive~$q$ provided that the polynomials coincide with the action of
the {\myem multiplier sequence}%
\footnote{See the definition and properties \emph{e.g.} in the book~\cite{Obreschkoff}. The
    sequence~$\big(q^{n(n-1)/2}\big)_{n=0}^{\infty}$ is formed by values of the
    function~$e^{\frac{1}{2}\ln |q|\cdot z(z-1)}$ at integer points; since~$\ln|q|<0$, it is a
    multiplier sequence according to Satz~10.1 from~\cite[p.~42]{Obreschkoff}.}%
~$\big(q^{n(n-1)/2}\big)_{n=0}^{\infty}$ on the polynomial~$(z+1)^N$. This justifies the
assertion of Conjecture~\ref{conjecture-C} for purely imaginary~$q$ without bounds on the ratio
of subsequent (by the absolute value) roots.

\paragraph*{Partial theta function.\hspace{-.5em}}
An analogous problem by A. Sokal appears in~\cite{KostovShapiro}. The partial theta function
\[
\Theta_0(z;q) = \sum_{n=0}^\infty q^{\frac{n(n-1)}2} z^{n}
\]
has only negative zeros whenever $0<q\le\widetilde q\approx 0.3092493386$, which is shown
in~\cite{KostovShapiro} (see also~\cite{Kostov}). Splitting it into the even part~$f(z^2)$ and
the odd part~$zg(z^2)$ gives
\begin{align*}
    \Theta_0(z;q)&=f(z^2)+zg(z^2),\\
    f(z)
    &=\sum_{n=0}^\infty q^{n(2n-1)} z^{n}
    =\sum_{n=0}^\infty q^{n(2n-2)} (qz)^{n}
    =\sum_{n=0}^\infty \left(q^4\right)^{\frac{n(n-1)}2} (qz)^{n} = \Theta_0(qz;q^4)
    \quad\text{and}\\
    g(z)
    &=\sum_{n=0}^\infty q^{n(2n+1)} z^{n}
    =\sum_{n=0}^\infty q^{n(2n-2)} (q^3z)^{n} = \Theta_0(q^3z;q^4).
\end{align*}
Thus, both~$f(z)$ and~$g(z)$ have only negative zeros whenever $0<q^4\le\widetilde q$. Therefore, by
Corollary~\ref{cr:ca2} all zeros of~$\Theta_0(z;iq)$ are simple and distinct in absolute value
if~$0<q^4\le\widetilde q$, that is if~$0<q\le q_*\approx 0.7457224107$. This is a partial
confirmation of the following assertion.
\begin{conjecture}[{see~\cite[p.~832]{KostovShapiro}}]
    Is it true that all zeros of~$\Theta_0(z;q)$ remain simple within the open
    disk~$|q| < \widetilde q$?
\end{conjecture}

With the help of exactly the same manipulations we could deduce that, for example, the (Jacobi)
theta function
\[
\Theta(z;iq) = \sum_{n=-\infty}^\infty (iq)^{\frac{n(n-1)}2} z^{n},\quad 0<q<1,
\]
also has its zeros simple, distinct in absolute value and residing in the quadrants of the
complex plane rotated by $\pi/4$ (according to the Remark~\ref{rem:akin_facts}). However, this
is redundant (although yet instructive) because the exact information is provided by the Jacobi
triple product formula (see \emph{e.g.}~\cite[Theorem~352]{HardyWright}) which is valid for any
complex $z\ne 0$ and~$|q|<1$:
\[
\sum_{n=-\infty}^\infty q^{\frac{n(n-1)}2}z^n =
\prod_{j=1}^{\infty} (1-q^{j}) (1+zq^{j-1}) \left(1+\frac {q^{j}}z\right).
\]

\section*{Acknowledgements}
The author is very grateful to the people who gave helpful comments concerning this study,
especially to the members (former and current) of his working group at the TU-Berlin. I also
thank the colleagues from Potsdam (Germany) and Ufa (Russian Federation) for useful discussions.
\vfill \eject

%\cohead{References}
{\setstretch{1}

}
\vspace{1em}
\noindent
\makebox[.295\textwidth][l]{\itshape Alexander Dyachenko}
\makebox[.7\textwidth][r]{\texttt{\href{mailto:diachenko@sfedu.ru}{diachenko@sfedu.ru},
        \href{mailto:dyachenk@math.tu-berlin.de}{dyachenk@math.tu-berlin.de}}}\\
\makebox[\textwidth][s]{\itshape TU-Berlin, Institut f\"ur Mathematik, Sekretariat MA 4-2,
    Straße des 17. Juni 136, 10623 Berlin, Germany}

\begin{thebibliography}{99}\itemsep0pt\parsep0pt\small
\bibitem{AESW} M. Aissen, A. Edrei, I. J. Schoenberg, A. Whitney, `On the generating functions
    of totally positive sequences', {\em Proc. Natl. Acad. Sci. USA} \textbf{37} (1951), 303--307.
\bibitem{Alander} M. \AA lander, {\em Sur le d\'eplacement des z\'eros des fonctions enti\`eres par
     leur d\'erivation}, Th\`ese, Almqvist  Wiksell, Uppsala, 1914 (French).
\bibitem{AndrewsBerndt} G.~E.~Andrews, B.~C.~Berndt, \emph{Ramanujan's Lost Notebook}. Part II,
    Springer, New York, 2009.% MR2474043
% \bibitem{AkhKrein} {M.~G.~Kre\u\i n}, `Concerning a special class of entire and
%     meromorphic functions', Article VI in N. I. Akhiezer and M. G. Kre\u\i n,
%     {\em Some questions in the theory of moments}, GNTIU, Kharkov, 1938 (Russian); English
%     translation by W. Fleming and D. Prill. {\em Translations of Mathematical Monographs},
%     Vol.~2. Amer. Math. Soc., Providence, R.I., 1962.
\bibitem{BarkovskyTyaglov} Y. Barkovsky and M. Tyaglov, `Hurwitz rational functions`,
    {\em Linear Algebra Appl.} \textbf{435} (2011), no.~8, 1845--1856.
% \bibitem{BRK_GS} B. C. Berndt, R. J. Evans and K. S. Williams, {\em Gauss and Jacobi sums},
%     Canadian Mathematical Society Series of Monographs and Advanced Texts, John Wiley \& Sons
%     Inc., New York, 1998.
\bibitem{Chebotarev} N. G. Chebotarev (N. Tschebotareff),
    `\"Uber die Realit\"at von Nullstellen ganzer transzendenter Funktionen',
    {\em Math. Ann.} \textbf{99} (1928), 660--686 (German).
\bibitem{ChebMei} {N.~G.~Chebotarev (aka \v Cebotar\"ev) and N. N. Me\u\i man},
    {\em The Routh--Hurwitz problem for polynomials and entire functions}, Trudy Mat. Inst.
    Steklova 26, Acad. Sci. USSR, Moscow--Leningrad, 1949.
%    (Russian).
    (Russian: {\otherlanguage{russian}Н.~Г.~Чеботарёв, Н.~Н.~Мейман,
        {\itshape Проблема Рауса-Гурвица для полиномов и целых функций}, Труды Мат. Инст.
        Стеклова XXVI, Изд. АН СССР, Москва--Ленинград, 1949.})
    %\ \href{http://mi.mathnet.ru/eng/tm1039}{\ttfamily http://mi.mathnet.ru/eng/tm1039}
% \bibitem{Duren} {P. Duren}, {\em Harmonic mappings in the plane}, Cambridge University Press,
%      NY, 2004.
\bibitem{Duren} {P. Duren}, {\em Univalent functions}, Springer-Verlag, New
    York-Berlin-Heidelberg-Tokyo, 1983.
\bibitem{Dyachenko1}{A. Dyachenko}, `On certain class of entire functions and a conjecture by
    Alan Sokal', \emph{preprint}, 2013, available online at
    \url{http://arxiv.org/pdf/1309.7551v1}
\bibitem{Edrei} A. Edrei, `On the generating function of a doubly infinite, totally positive
    sequence', {\em Trans. Amer. Math. Soc.} \textbf{74} (1953), no.~3, 367--383.
\bibitem{Eremenko_Ostrovsky} A.~Eremenko and I.~Ostrovsky, `On the ``pits effect'' of Littlewood
    and Offord', {\em Bull. Lond. Math. Soc.} \textbf{39} (2007), no.~6, 929--939.
\bibitem{Kostov} V. P. Kostov, `On the zeros of a partial theta function', {\em Bull. Sci. math.}
    \textbf{137} (2013), 1018--1030.
\bibitem{KostovShapiro} V. P. Kostov and B.~Shapiro, `Hardy–Petrovitch–Hutchinson’s problem and
    partial theta function', {\em Duke Math. J.} \textbf{162} (2013), no.~5, 825--861.
\bibitem{Gantmakher}{F.~R.~Gantmacher (aka Gantmakher)}, {\em The theory of matrices},
    Vol. 2, Translated by K.~A.~Hirsch, Chelsea Publishing Co., New York, 1959.
\bibitem{HardyWright} G. H. Hardy and E. M. Wright,
    {\em An Introduction to the Theory of Numbers} (Fourth ed.), Clarendon Press, Oxford, 1975.
\bibitem{Langley} J.K. Langley, `A certain functional-differential equation', {\em J. Math. Anal.
    Appl.} \textbf{244} (2000), no.~2, 564--567
\bibitem{Liu} Y.-K. Liu, `On some conjectures by Morris et al. about zeros of an entire
    function', {\em J. Math. Anal. Appl.} \textbf{226} (1998), no.~1, 1--5.
\bibitem{Morris_et_al} G. Morris, A. Feldstein and E. Bowen, `The Phragmen-Lindel\"of principle
    and a class of functional differential equations',
    {\em Ordinary differ. Equat., Proc. NRL-MRC Conf., Washington D.C. 1971}, Academic Press,
    New York, 1972, 513--540.
\bibitem{Obreschkoff} N. Obreschkoff,
    {\em Verteilung und Berechnung der Nullstellen reeller Polynome}, VEB Deutscher
    Verlag der Wissenschaften, Berlin, 1963  (German).
\bibitem{Postnikov} M.~M.~Postnikov,
    {\em Stable polynomials}, Nauka, Moscow, 1981.
    % (Russian).
    (Russian: {\otherlanguage{russian}М. М. Постников, {\em Устойчивые многочлены}, Наука, М., 1981.})   
\bibitem{Sokal}A.~D.~Sokal, `Conjectures on the function
    $F (x, y) = \sum_{n=0}^\infty x^n y^{n(n-1)/2} /n!$ the polynomials
    $\relpenalty=100 P_N (x, w) = \sum_{n=0}^N\binom{N}{n}x^nw^{n(N-n)}$, and the generating polynomials of
    connected graphs', \emph{in preparation}. A fairly detailed summary can be found at
    \url{http://ipht.cea.fr/statcomb2009/misc/Sokal_20091109.pdf} and at
    \url{http://www.maths.qmul.ac.uk/~pjc/csgnotes/sokal/}
\bibitem{Sokal_theta}A.~D.~Sokal, `The leading root of the partial theta function',
    {\em Adv. in Math.} \textbf{229} (2012), no.~5, 2603--2621.
\bibitem{Tyaglov} M.~Tyaglov, `Generalized Hurwitz polynomials', \emph{pre\-print}, 2010,
     available online at\\\url{http://arxiv.org/pdf/1005.3032v1.pdf}
% \bibitem{TyaglovDyachenko} M.~Tyaglov, A.~Dyachenko, in preparation.
\bibitem{Valiron} G.~Valiron,
    `Sur une \'equation fonctionnelle et certaines suites de facteurs',
    {\em J. Math. Pures Appl.}, IX. S\'er. \textbf{17} (1938), 405--423.
\bibitem{Warnaar} S.~O.~Warnaar, `Partial theta functions. I. Beyond the lost notebook', Proc.
    London Math. Soc., 3rd ser. \textbf{87} (2003), No.~2,
    363--395. % {doi:10.1112/S002461150201403X}
\bibitem{Wolff} J. Wolff, `L'int\'egrale d'une fonction holomorphe et \`a partie r\'eelle
    positive dans un demi-plan est univalente', {\em C. R.} \textbf{198} (1934), 1209--1210.
    % http://gallica.bnf.fr/ark:/12148/bpt6k31506.image
\end{thebibliography}
\end{document}